\documentclass[11pt]{amsart}
\usepackage{MnSymbol}
\usepackage{mathrsfs}
\usepackage[utf8]{inputenc}
\usepackage{amsmath,amsthm}
\usepackage{tikz}
\usetikzlibrary{matrix}

\newtheorem{Theorem}{Theorem}[section]
\newtheorem{Corollary}[Theorem]{Corollary}
\newtheorem{Lemma}[Theorem]{Lemma}
\newtheorem{Proposition}[Theorem]{Proposition}
\theoremstyle{definition}
\newtheorem{Remark}[Theorem]{Remark}
\newtheorem{Example}[Theorem]{Example}

\numberwithin{equation}{section}

\DeclareMathAlphabet\mathbb{U}{msb}{m}{n}
\newcommand{\mono}{\rightarrowtail}
\newcommand{\epi}{\twoheadrightarrow}
\newcommand{\ilimit}{\,\varprojlim{}\:}
\def\aa{{\mathfrak{a}}}
\def\LL{\Lambda}

\usepackage[all]{xy}
\begin{document}

\title{On Bousfield problem for the class of metabelian groups}
\author{Sergei O. Ivanov}
\address{Chebyshev Laboratory, St. Petersburg State University, 14th Line, 29b,
Saint Petersburg, 199178 Russia} \email{ivanov.s.o.1986@gmail.com}

\author{Roman Mikhailov}
\address{Chebyshev Laboratory, St. Petersburg State University, 14th Line, 29b,
Saint Petersburg, 199178 Russia and St. Petersburg Department of
Steklov Mathematical Institute} \email{rmikhailov@mail.ru}

\begin{abstract} The homological properties of localizations and
completions of metabelian groups are studied. It is shown that,
for $R=\mathbb Q$ or $R=\mathbb Z/n$ and a finitely presented
metabelian group $G$, the natural map from $G$ to its
$R$-completion induces an epimorphism of homology groups
$H_2(-,R)$. This answers a problem of A.K. Bousfield for the class
of metabelian groups.
\end{abstract}

\thanks{This research is supported by the Chebyshev
Laboratory  (Department of Mathematics and Mechanics, St.
Petersburg State University)  under RF Government grant
11.G34.31.0026 and by JSC "Gazprom Neft", as well as by the RF
Presidential grant MD-381.2014.1. The first author is supported by
RFBR (grant no. 12-01-31100 mol\_a, 13-01-00902 A)} \maketitle

\section{Introduction}
The subject of investigation of this paper is the relation between
inverse limits of groups and the second homology $H_2(-,K)$ for
certain coefficients $K$. One of the results of the paper is the
following. Let $G$ be a finitely presented metabelian group,
$\{\gamma_i(G)\}_{i\geq 0}$ its lower central series, then, for
any $n>0$, there is a natural isomorphism
$$
H_2(\ilimit\ G/\gamma_i(G),\mathbb Z/n)\simeq \ilimit\
H_2(G/\gamma_i(G),\mathbb Z/n).
$$
That is, in this particular case, the inverse limit commutes with
the second homology functor.

The problem of relation between inverse limit and second homology
of groups appears in different areas of algebra and topology.
Recall two related open problems, one from \cite{Bousfield}, the
second from \cite{FKRS}. A.K. Bousfield posed the following question in~\cite{Bousfield}, Problem 4.10:\\ \\
{\bf Problem.}\ {\em (Bousfield)}  Is $E^RX\to \hat X_R$ iso when
$X$ is a finitely presented group when $R=\mathbb Q$ or
$R=\mathbb Z/n$?\\

In the above problem, $E^R$ is the $HR$-localization functor
defined in~\cite{Bousfield} and $\hat X_R$ is the $R$-completion
of the group $X$. It follows immediately from the construction of
$E^R$, that, for a finitely presented group $X$, the map $E^RX\to
\hat X_R$ is isomorphism if and only if the completion map $X\to
\hat X_R$ induces an epimorphism $H_2(X;R)\to H_2(\hat X_R, R)$.
In this paper we prove the following (see Corollary
\ref{cor_bousfield})

\vspace{.5cm}\noindent{\bf Theorem.} For a finitely presented
metabelian group $X$, the natural map $E^RX\to \hat X_R$ is an
isomorphism for
$R=\mathbb Q$ or $R=\mathbb Z/n$. \\

Observe that, the above result can not be generalized to the case
$R=\mathbb Z$ as the following simple example shows. For the Klein
bottle group $G=\langle a,b\ |\ a^{-1}bab=1\rangle$, the second
homology $H_2(\hat G,\mathbb Z)$ is isomorphic to the exterior
square of the 2-adic integers and therefore is uncountable (see
\cite{Bousfield}).\\

The second problem of a similar type is the question of comparison
between discrete and continuous homology of pro-$p$-groups. For
any pro-p-group $P$, one can look at homology of $P$ in two
different ways: as homology of a discrete group and as homology of
a topological group. There are natural comparison maps between
these homology groups (see \cite{FKRS} for detailed discussion):
$$
\phi_n: H_n^{discrete}(P,\mathbb Z/p)\to H_n^{cont}(P, \mathbb
Z/p)
$$
Analogously, there are maps between cohomology groups $\phi^n:
H_{cont}^n(P,\mathbb Z/p)\to H_{discrete}^n(P,\mathbb Z/p)$.
In~\cite{FKRS}, G.A. Fernandez-Alcober, I.V. Kazachkov, V.N. Remeslennikov, and P. Symonds asked the following question.\\ \\
{\bf Problem.}\ Does there exist a finitely presented pro-$p$
group $P$ for which $\phi^2 : H^2_{\text{cont}} (P,\mathbb Z/p)
\to H^2_{\text{discrete}}(P,\mathbb Z/p)$ is not an isomorphism?\\

See \cite{Serre} for background on the continuous cohomology of
pro-$p$ groups. It is shown in \cite{FKRS} that, for a finitely
presented pro-$p$-group $P$, the following two conditions are
equivalent:\\ \\
(1) the map $\phi^2: H^2_{\text{cont}}(P,\mathbb Z/p)\to
H^2_{\text{discrete}}(P,\mathbb Z/p)$ is an isomorphism;\\ \\
(2) the map $\phi_2: H_2^{\text{discrete}}(P,\mathbb Z/p)\to
H_2^{\text{cont}}(P,\mathbb Z/p)$ is an isomorphism.

Our contribution to this problem is the following: for a finitely
presented metabelian group $G$, there is a natural isomorphism
(see Corollary \ref{cor_limit})
$$
H_2^{\text{discrete}}(\hat G_p,\mathbb
Z/p)\buildrel{\simeq}\over\to H_2^{\text{cont}}(\hat G_p,\mathbb
Z/p)
$$

Observe that, it is not possible to extend the above results to
the class of all finitely generated metabelian groups. For
example, one can consider the $p$-Lamplighter group
$$
L_p:=\mathbb Z/p\wr \mathbb Z=\langle a,b\ |\ a^p=[a,a^{b^i}]=1,\
i\in \mathbb Z\rangle
$$
It is easy to check that the second homology of the
pro-p-completion $H_2(\hat L_p,\mathbb Z/p)$ is an uncountable
$\mathbb Z/p$-vector space. Despite the fact that all finitely
generated metabelian groups are subgroups of finitely presented
metabelian groups (see \cite{Baumslag1} and \cite{Rem}), the
finite presentability is a crucial point in the results like
Corollary \ref{cor_limit} and Corollary \ref{cor_bousfield}.

Let $G$ be a finitely generated metabelian group with a metabelian
decomposition $M\mono G \epi A$. The group $G$ is finitely
presented if and only if the $\mathbb Z[A]$-module $M$ is tame in
the sense of Bieri-Strebel \cite{Bieri-Strebel_1981}. Tame modules
are characterized via geometric properties (see section 7) and
certain group-theoretic properties of finitely generated
metabelian groups can be formulated in the language of commutative
algebra. In this paper, the properties of tame modules are used
for investigation of homological behavior of $R$-completions.

The paper is organized as follows. In section 2 we recall the
general properties of inverse limits of groups as well as the
properties of their derived functors. In section 3 we recall the
definitions and properties of $I$-adic completions of modules and
describe the structure of a {\it twisted exterior square} of a
completed and localized module. The structure of the twisted
exterior square plays a central role in the study of the second
homology of the completed and localized metabelian groups. Section
4 is about $R$-completions of finitely generated metabelian
groups, where $R$ is either a subring of rationals or a finite
ring $\mathbb Z/n$. A natural way to compare homology of a
metabelian group with homology of its $R$-completion it to
consider the standard homology spectral sequences for
corresponding metabelian decompositions. Sections 5,6 contain
technical properties of functors which appear in these spectral
sequences. Section 7 is a technical section, where the finite
presentability appears. The properties of exterior squares of tame
modules in the sense of Bieri-Strebel play a key role in the whole
picture (see Proposition \ref{prop_wedge_compl_loc}). The main
results of the paper are essentially based on the fact that one
can control the $E_{0,2}^2$-term of the standard homology spectral
sequence for a metabelian decomposition of $R$-completions.

The main results of the paper are theorems \ref{theorem_limit} and
\ref{Theorem_Bousfield}.  These theorems are proved in sections 8
and 9. Theorem \ref{theorem_limit} is formulated as follows. Let
$G$ be a finitely presented metabelian group, $R=\mathbb
Z[J^{-1}]$ or $R=\mathbb Z/n$ and $K$ be an Artinian quotient ring
of $R$. Then the $R$-lower central quotient maps induce the
isomorphisms
$$H_2(\hat G_R,K)\cong \ilimit H_2(G/\gamma^R_{i}(G),K).$$
As a simple corollary of this theorem, we get the natural
isomorphism between discrete and continuous homology groups of
pro-$p$-completions of finitely presented groups, which gives an
answer to a particular case of the problem from \cite{FKRS}
mentioned above. Theorem \ref{Theorem_Bousfield} is the following.
Assuming that $G$, $R$ and $K$ are as in theorem
\ref{theorem_limit}, for $i>\!>0$, there is a short exact sequence
$$0 \longrightarrow \Phi^R_iH_2(G,K)\longrightarrow
H_2(G,K)\longrightarrow H_2(\hat G_R,K) \longrightarrow 0,$$ where
the epimorphism is induced by the homomorphism $G\to \hat G_R.$
Here $\Phi^R_iH_2(G,K)$ is the $i$th term of an $R$-analog of the
Dwyer filtration of $H_2(G,K)$. That is, the kernel of the map of
$H_2(-,K)$ induced by $R$-completion, is described.

In the final section 10 we consider the second homology with
finite coefficients of the Telescope functor (see
\cite{Baumslag-Mikhailov-Orr} for properties and applications of
the Telescope). We prove that, for a finitely presented metabelian
group $G$, the inclusion of the Telescope of $G$ into pronilpotent
completion $\bar G\hookrightarrow \hat G$ induces a natural
isomorphism $H_2(\bar G, \mathbb Z/n)\cong H_2(\hat G, \mathbb
Z/n)$.

Recall that, for a free group $F$ of rank 2, the second homology
$H_2(\hat F,\mathbb Z)$ is uncountable \cite{Bousfield}. The proof
is based on the construction of a free simplicial resolution of
the Klein bottle group and uses the homology spectral sequence for
a bisimplicial group. Observe that, the results of this paper show
that the same type of proof can not be used for homology with
finite coefficients. This motivates the following:\\ \\ {\bf
Problem.} {\it Is it true that, for every $n\geq 1$ and a free
group $F$, $H_2(\hat F, \mathbb Z/n)=0$?}

\section{Inverse limits of abelian groups and modules.}
Denote by $\omega$ the category, whose objects are natural numbers and
$$\omega(n,m)=\left\{
\begin{array}{ll}
\{(n,m)\}, & \text{ if } n\leq m \\
\emptyset, & \text{ if } n>m.
\end{array}\right.
$$ Then an inverse sequence of abelian groups
$$A_0\longleftarrow A_1 \longleftarrow A_2 \longleftarrow \dots$$
can be considered as a functor $A:\omega^{\rm op}\to {\sf Ab}.$
Consider the functor category ${\sf Ab}^{\omega^{\rm op}}.$ Since,
the category $\omega^{\rm op}$ is a free category generated by a
graph, the global dimension of the abelian category ${\sf
Ab}^{\omega^{\rm op}}$ is equal to $1$. The functor $$\varprojlim:
{\sf Ab}^{\omega^{\rm op}}\to {\sf Ab}$$ is left exact and its
derived functors we denote by $\varprojlim^i={\bf
R}^i\varprojlim.$ Since the global dimension is equal to $1$, we
get $\varprojlim^i=0$ for $i\geq 2.$ Moreover, if the
Mittag-Leffler condition holds for the inverse sequence $\{A_i\}$,
we have ${\varprojlim}^1 A_i=0.$ For example, the Mittag-Leffler
condition holds if the homomorphisms $A_{i+1}\to A_i$ are
epimorphisms.

Let $C^\bullet_i$ be an inverse sequence of (not necessarily bounded) complexes of abelian groups:
$$C^\bullet_0\longleftarrow C^\bullet_1 \longleftarrow C^\bullet_2 \longleftarrow \dots .$$
It can be considered as a complex
 $C^\bullet$ in the abelian category
  ${\sf Ab}^{\omega^{\rm op}}.$
  Let $I^{\bullet\bullet}$ be a right Cartan-Eilenberg resolution of $C^\bullet$ in ${\sf Ab}^{\omega^{\rm op}}$.
  Since the global dimension is $1$, we can chose $I^{\bullet\bullet}$ so that $I^{\bullet i}=0$ for $i\geq 2.$ It follows that the
  totalisations are equal ${\sf Tot}^{\oplus}(I^{\bullet\bullet})={\sf Tot}^{\prod}(I^{\bullet\bullet})$ and we denote it by ${\sf Tot}(I^{\bullet\bullet}).$
  Then we have
$${\bf R}^i\varprojlim(C^\bullet)= H^i( \ilimit {\sf Tot}(I^{\bullet\bullet}))=H^i({\sf Tot}( \ilimit I^{\bullet\bullet})),$$
where ${\bf R}^i\varprojlim$ is the right hyper derived functor of
$\varprojlim$ (see \cite[5.7.9]{Weibel}).

\begin{Lemma}\label{lemma_spectral_sequences_limit}
There are two spectral sequences $E$ and $'\!E$ such that
$$E_2^{pq}=H^p(\underset{i}\varprojlim^q\  C^\bullet_i)\
\Rightarrow\ {\bf R}^{p+q}\varprojlim(C^{\bullet}),$$

$$'\!E_2^{pq}=\underset{i}\varprojlim{}^p  \ H^q(C^\bullet_i)\ \Rightarrow\  {\bf R}^{p+q}\varprojlim(C^{\bullet}),$$
and $E^{pq}_1=\varprojlim^q C^p_i.$
\end{Lemma}
\begin{proof}
The double complex $\varprojlim I^{\bullet\bullet}$ has only two
nonzero rows, and hence the canonical filtrations of $\varprojlim
{\sf Tot}(I^{\bullet\bullet})$ are bounded. It follows that the
both sequences of a double complex $E$ and $'\!E$ converge to
${\bf R}^{p+q}\varprojlim(C^{\bullet}).$ Further, as in
{\cite[5.7.9]{Weibel}}, we get $E_1^{pq}=\varprojlim^q\:  C^p_i,$
$E_2^{pq}=H^p(\varprojlim^q\:  C^\bullet_i)$ and
$'\!E_2^{pq}=\varprojlim{}^p \: H^q(C^\bullet_i).$
\end{proof}
\begin{Remark}
The difference between the general statement in {\cite[5.7.9]{Weibel}} and Lemma \ref{lemma_spectral_sequences_limit} is that in our case all the spectral sequences converge in the strict sense. This lemma can be proved for any left exact functor from an abelian category of finite global dimension.
\end{Remark}
\begin{Corollary}\label{corollary_shoert_eact _sequence_limit}
If morphisms $C^\bullet_{i+1}\to C^\bullet_i$ satisfy the
Mittag-Leffler condition, then there is a short exact sequence
$$0\longrightarrow \underset{i}{\varprojlim}^1  \
H^{q-1}(C^\bullet_i)\longrightarrow H^q(\underset{i}\varprojlim\
C_i^\bullet) \longrightarrow  \underset{i}\varprojlim  \
H^{q}(C^\bullet_i)\longrightarrow 0.$$
\end{Corollary}
\begin{Proposition}\label{prop_tensor_tor_abelian_limit}
Let $\{A_i\}$ be an inverse sequence of abelian groups and $B$ a finitely generated abelian group. Then the natural morphisms are isomorphisms
$$\varprojlim {\sf Tor}(A_i,B)\cong {\sf Tor}(\varprojlim A_i,B), \hspace{1cm}({\varprojlim}^1  A_i)\otimes B\cong {\varprojlim}^1 (A_i\otimes B)$$
and for the natural morphisms $$\alpha: (\varprojlim A_i)\otimes
B\to \varprojlim \: (A_i\otimes
B),\hspace{1cm}\beta:{\varprojlim}^1\: {\sf Tor}(A_i,B)\to {\sf
Tor}({\varprojlim}^1 A_i,B)$$ there are natural isomorphisms
$${\rm Ker}(\alpha) \cong {\rm Ker}(\beta),\hspace{1cm} {\rm
Coker}(\alpha) \cong {\rm Coker}(\beta).$$
\end{Proposition}
\begin{proof}
Let $0\to P_1\to P_0 \to B\to 0$ be a free  presentation of $B,$ where $P_0,P_1$ are finitely generated free abelian groups. Consider the acyclic complex
$$C^\bullet=(\dots \to 0\to {\sf Tor}(A_i,B)\to A_i\otimes P_1 \to A_i\otimes P_0 \to A_i\otimes B \to 0 \to \dots ).$$ Then by Lemma
\ref{lemma_spectral_sequences_limit} we get a spectral sequence
$E$ that converges to zero and $E_1^{pq}={\varprojlim}^q C_i^p$.
The first page $E_1$ looks as follows
$$\begin{tikzpicture}
  \matrix (m)
 [matrix of math nodes,row sep=1cm,column sep=1cm,minimum width=2em]
  { {\varprojlim}^1 {\sf Tor}(A_i,B) &    {\varprojlim}^1 (A_i\otimes P_1) & {\varprojlim}^1 (A_i\otimes P_0) & {\varprojlim}^1 (A_i\otimes B)\\
 {\varprojlim}\: {\sf Tor}(A_i,B) &    {\varprojlim}\: (A_i\otimes P_1) & {\varprojlim}\: (A_i\otimes P_0)& {\varprojlim}\: (A_i\otimes B). \\};
  \path[-stealth]
(m-1-1.east|-m-1-2) edge  (m-1-2)
(m-2-1.east|-m-2-2) edge  (m-2-2)
(m-1-2.east|-m-1-3) edge node [above] {$f^1$}  (m-1-3)
(m-2-2.east|-m-2-3) edge node [above] {$f^0$}  (m-2-3)
(m-1-3.east|-m-1-4) edge  (m-1-4)
(m-2-3.east|-m-2-4) edge  (m-2-4)
;
\end{tikzpicture}$$
Since $P_j$ is a finitely generated abelian group and $\varprojlim^q$ is an additive functor, we get $\varprojlim^q(A_i\otimes P_j)\cong (\varprojlim^q A_i)\otimes P_j.$ It follows that
${\rm Ker}(f^q)={\sf Tor}(\varprojlim^q A_i,B)$ and ${\rm Coker}(f^q)=(\varprojlim^q A_i)\otimes B.$ We can replace the middle four terms in the spectral sequence with the kernels and cokernels of $f^q$ so that the new spectral sequence still converges to zero:
$$\begin{tikzpicture}
  \matrix (m)
 [matrix of math nodes,row sep=1cm,column sep=1cm,minimum width=2em]
  { {\varprojlim}^1 {\sf Tor}(A_i,B) &    {\sf Tor}({\varprojlim}^1 A_i,B) & ({\varprojlim}^1 A_i)\otimes B & {\varprojlim}^1 (A_i\otimes B)\\
 {\varprojlim}\: {\sf Tor}(A_i,B) &    {\sf Tor}({\varprojlim}\: A_i, B) & ({\varprojlim}\: A_i)\otimes B& {\varprojlim}\: (A_i\otimes B). \\};
  \path[-stealth]
(m-1-1.east|-m-1-2) edge node [above] {$\beta$}  (m-1-2)
(m-2-1.east|-m-2-2) edge  (m-2-2)
(m-1-2.east|-m-1-3) edge node [above] {$0$} (m-1-3)
(m-2-2.east|-m-2-3) edge node [above] {$0$}  (m-2-3)
(m-1-3.east|-m-1-4) edge  (m-1-4)
(m-2-3.east|-m-2-4) edge node [above] {$\alpha$}  (m-2-4)
;
\end{tikzpicture}$$
Analysing the second page of this spectral sequence we obtain the required isomorphisms.
\end{proof}
\begin{Corollary} Let $\{A_i\}$ be an inverse sequence of abelian groups that satisfies the Mittag-Leffler condition. Then there is the following short exact sequence
$$0\longrightarrow {\varprojlim}^1 {\sf Tor}(A_i,B)\longrightarrow (\varprojlim \:A_i)\otimes B \overset{\alpha}\longrightarrow \varprojlim\: (A_i\otimes B) \longrightarrow 0.$$
\end{Corollary}
The following proposition is a generalisation of the previous corollary.

\begin{Proposition}\label{proposition_tor_lim_gen}
Let $\LL$ be an associative right Notherian ring, $M$ a finitely
generated right $\LL$-module and  $\{N_i\}$ an inverse sequence of
left $\LL$-modules that satisfies the Mittag-Leffler condition.
Then, for $m\geq 0,$ there is the following short exact sequence
$$0\longrightarrow {\varprojlim}^1 {\sf Tor}_{m+1}^{\LL}(M,N_i)\longrightarrow {\sf Tor}_m^{\LL}(M,\varprojlim \: N_i) \longrightarrow \varprojlim\: {\sf Tor}_m^{\LL}(M, N_i) \longrightarrow 0.$$
\end{Proposition}
\begin{proof} Since, $\LL$ is right Notherian and $M$ is finitely generated, there exists a projective resolution $P_\bullet$ of $M$ that consists of finitely generated projective modules. Let $C^\bullet_i$ be the complex $P_\bullet\otimes_{\LL}N_i.$ The Mittag-Leffler condition for $N_i$ implies the Mittag-Leffler condition for $P_\bullet \otimes_\LL N_i$.  Then by Corollary \ref{corollary_shoert_eact _sequence_limit} we get the short exact sequence
$$ 0\longrightarrow {\varprojlim}^1 {\sf Tor}_{m+1}^\LL(M,N_i) \longrightarrow H_m(\varprojlim\: (P_\bullet \otimes_{\LL} N_i))\longrightarrow \varprojlim\: {\sf Tor}_m^\LL(M,N_i)\longrightarrow 0.$$
Since $P_\bullet$ consists of finitely generated projective
modules, we get the isomorphism $$\varprojlim\: (P_\bullet
\otimes_{\LL} N_i)\cong  P_\bullet \otimes_{\LL} (\varprojlim\:
N_i).$$ Thus the middle term is isomorphic to ${\sf
Tor}^{\LL}_m(M,\ilimit N_i)$
\end{proof}

\section{Completion and localization of rings and modules.}\label{section_rings}
First we remind main concepts concerned to $I$-adic topology \cite[VIII]{Zariski-Samuel_II}, \cite[III]{Bourbaki}.
Throughout the section all rings are assumed to be Noetherian and commutative. Let $I$ be an ideal of a Noetherian commutative  ring $\LL.$
 A (right) $\Lambda$-module we endow by the $I$-adic topology i.e. the topology such that the submodules $\{MI^n\}$ form a fundamental system of neighbourhoods of zero. In particular, the ring $\Lambda $ is endowed by the $I$-adic topology. The closure of a submodule $N\leq M$ is given by  ${\sf cl}(N)=\bigcap (N+MI^n).$ The submodule $N$ is open if and only if  $N\supseteq MI^n$ for some $n.$ We put $MI^{\infty}:=\bigcap MI^i={\sf cl}(0).$ The module $M$ is said to be nilpotent if $MI^i=0$ for $i>\!>0.$ The module $M$ is said to be residually nilpotent if $MI^{\infty}=0.$ The module is residually nilpotent if and only if it is Hausdorff in the $I$-adic topology. The ideal $I$ has the Artin-Rees property i.e. for a finitely generated module $M$ and its submodule $N$ the $I$-adic topology on the $N$ coincides with the induced topology. In particular,
\begin{equation}\label{eq_artin-rees-1}
MI^{\infty}\cdot I=MI^{\infty}.
\end{equation}
We put
$$M_{\sf rn}:=M/MI^{\infty}.$$
Then the projection $M\to M_{\sf rn}$ is the universal homomorphism from $M$ to a residually nilpotent module.

The $R$-completion of $M$ is the inverse limit $\hat M= \hat M_I=\ilimit M/MI^i$ with the natural structure of $\hat \LL=\ilimit \LL/I^i$-module and the natural $\LL$-homomorphism $\varphi_M:M\to \hat M.$ The ring $\hat \LL$ is still Noetherian, the morphism $\varphi:\LL\to \hat \LL$ is flat and for a finitely generated $\LL$-module $M$ there is an isomorphism $\hat M\cong M\otimes_{\LL} \hat \LL.$  The ideal $\hat I={\rm Ker}(\hat \LL\to \LL/I)$ is equal to  $\hat \LL\cdot\varphi(I),$ and there are isomorphisms $\LL/I^n\cong \hat\LL/\hat I^n.$

The notion of the $I$-adic completion is related to the notion of
localization by the multiplicative set $1+I.$ We put
$\LL^\ell=\LL^\ell_I=\LL[(1+I)^{-1}]$ and
$M^{\ell}=M^{\ell}_{I}=M[(1+I)^{-1}].$ It is well-known that the
morphism $\LL\to\LL^\ell$ is  flat and $M^\ell=M\otimes_{\LL}
\LL^\ell.$ Moreover, if we denote $I^\ell=I[(1+I)^{-1}],$ then
$\LL/I^n\cong \LL^\ell/(I^\ell)^n.$  Since every element of
$1+\hat I$ is invertible $\big((1+x)^{-1}=\sum_{i=0}^{\infty}
(-1)^i x^i\big),$ the morphism $\varphi:\LL\to \hat\LL$ lifts to
the morphism $\LL^{\ell}\to \hat \LL.$

 The ring $\LL$ is said to be {\it Zariski ring} with respect to the ideal $I$ if one of the following equivalent properties
 holds (see \cite[VII \S 4]{Zariski-Samuel_II}):
\begin{itemize}
\item every submodule of every  finitely generated $\LL$-module is closed;
\item every finitely generated $\LL$-module $M$ is residually nilpotent;
\item every ideal of $\LL$ is closed;
\item every element of $1+I$ is invertible.
\end{itemize}
 For any ring $\LL$ and an ideal $I\triangleleft \LL$ there are two constructions that give examples of Zariski rings: the $I$-adic completion $\hat\LL$ and the localization $\LL^\ell$ by the set $1+I$.  We are interested in both of these situations, so we work in the following general case.

Consider a ring homomorphism  $\varphi: \LL\to \tilde \LL$ that satisfies the following conditions:
\begin{flalign}\label{eq_conditions_Zar}
\begin{array}{l}
1)\ \tilde \LL \text{ is a Zariski ring with respect to the ideal } \tilde I=\varphi(I)\cdot \tilde \LL; \\
2)\ \varphi \text{ is flat i.e. } \tilde \LL \text{ is a flat }\LL\text{-module;}\\
3)\ \varphi \text{ induces the isomorphism } \LL/I^n\cong \tilde \LL/\tilde I^n \text{ for }  n\geq 0.
 \end{array}
\end{flalign}
We assume that $\tilde \LL$ is endowed by the $\tilde I$-adic
topology.  It is easy to see that  $\varphi$ is continuous,
$\varphi(\LL)$ is dense in $\tilde \LL$ and $\tilde
I^n=\varphi(I^n)\tilde \LL.$  For a  $\LL$-module $M$ we set
$\tilde M=M\tilde{\ }=M\otimes_{\LL} \tilde \LL$ and $\varphi_M=
1\otimes \varphi:M\to \tilde M.$ Then the functor
$\tilde{(-)}:{\sf Mod}(\LL)\to {\sf Mod}(\tilde \LL)$ is exact.
The sequence of isomorphisms $$\tilde M/ \tilde M\tilde I^n\cong
\tilde M  \otimes_{\tilde{\LL}} \tilde \LL/\tilde I^n\cong
M\otimes_{\LL}\tilde{\LL} \otimes_{\tilde \LL} \tilde \LL/\tilde
I^n \cong\\ M\otimes_{\LL} \tilde \LL/\tilde I^n \cong
M\otimes_\LL \LL/I^n \cong M/MI^n$$ together with $\tilde
MI^n=\tilde M\cdot \varphi(I^n)=\tilde M\cdot \tilde \LL\cdot
\varphi(I^n)=\tilde M\tilde I^n$ give isomorphisms
\begin{equation}\label{eq_coinvariants}
M/MI^n\cong \tilde M/\tilde M\tilde I^n\cong \tilde M/\tilde MI^n
\end{equation}
for any $\LL$-module $M$ and $n\geq 0.$ It follows that $\tilde M=\varphi_M(M)+\tilde M\tilde I^n.$ Since, every finitely generated $\tilde \LL$-module is residually nilpotent, we get an isomorphism
$$(M_{\sf rn})^{\sim}\cong M^{\sim}.$$ It follows that the morphism $\varphi_M:M\to \tilde M$ is the composition of morphisms $M\to M_{\sf rn}\to \tilde M.$

\begin{Lemma}\label{lemma_tor} Let $f:\Gamma \to \LL$ be a ring homomorphism, $M,N$ be $\LL$-modules, $X$ be a $\Gamma$-module and $i,j\geq 0$. Then there are the following isomorphisms.
\begin{enumerate}
\item ${\sf Tor}^{\Gamma}_i(\tilde M,X)\cong {\sf Tor}_i^{\Gamma}(M,X)^{\sim}.$
\item ${\sf Tor}^{\LL}_i(\tilde M,N)\cong {\sf Tor}^{\LL}_i(M,\tilde N)\cong  {\sf Tor}_i^{\LL}(M,N)^{\sim}.$
\item
The morphisms $M\to M_{\sf rn} \to \tilde M$ induce isomorphisms  $${\sf Tor}_i^\LL(\LL/I,{\sf Tor}^{\Gamma}_j( M,X))\cong {\sf Tor}^{\LL}_i(\LL/I,{\sf Tor}^{\Gamma}_j(M_{\sf rn},X))\cong {\sf Tor}^{\LL}_i(\LL/I,{\sf Tor}^{\Gamma}_j(\tilde M,X)) $$
\end{enumerate}
\end{Lemma}
\begin{proof}
(1) Let  $P_\bullet$ be a finitely generated  $\Gamma$-projective
resolution of $X.$ Then we have \begin{multline*} {\sf
Tor}_*^{\Gamma}(\tilde M,X)= H_*(\tilde M\otimes_{\Gamma}
P_\bullet)\cong H_*(\tilde \LL\otimes_\LL M \otimes_\Gamma
P_\bullet )\cong\\ \tilde \LL\otimes_\LL H_*( M\otimes_\Gamma
P_\bullet)\cong {\sf Tor}_*^{\Gamma}(M,X)^{\sim}.\end{multline*}

(2) It follows from the previous formula and the isomorphism ${\sf Tor}_*(M,N)\cong {\sf Tor}_*(N,M).$

(3) Using the previous isomorphisms and $(\LL/I)^{\sim}\cong
\LL/I$, we get \begin{multline*}{\sf Tor}^{\LL}_i(\LL/I,{\sf
Tor}^{\Gamma}_j(\tilde M,X))\cong {\sf Tor}^{\LL}_i(\LL/I,{\sf
Tor}^{\Gamma}_j( M,X)^{\sim})\cong\\ {\sf
Tor}^{\LL}_i((\LL/I)^{\sim},{\sf Tor}^{\Gamma}_j( M,X))\cong {\sf
Tor}^{\LL}_i(\LL/I,{\sf Tor}^{\Gamma}_j( M,X))\end{multline*} and
similarly
 $${\sf Tor}^{\LL}_i(\LL/I,{\sf Tor}^{\Gamma}_j( (M_{\sf rn})^{\sim},X))\cong  {\sf Tor}^{\LL}_i(\LL/I,{\sf Tor}^{\Gamma}_j(M_{\sf
 rn},X)).$$
 Using that $(M_{\sf rn})^{\sim}\cong \tilde  M,$ we obtain the required isomorphisms.
\end{proof}

Since $\tilde \LL$ is a Zariski ring, the annihilator ${\sf Ann}_{\tilde \LL}\: \tilde M$ is a closed ideal of $\tilde \LL$ and hence
\begin{equation}\label{eq_incl_cl_ann}
{\sf cl}({\sf Ann}_\LL\: M) \subseteq {\sf Ann}_\LL\:\tilde M.
\end{equation}
\begin{Lemma}\label{lemma_ses}If ${\sf cl}({\sf Ann}_{\LL}\: M) \supseteq I^{n}$ then $MI^n=MI^{n+1}$  and there is the short exact sequence
$$0\longrightarrow  MI^n \longrightarrow M \overset{\varphi_M}\longrightarrow \tilde M \longrightarrow 0.$$  In particular, $\tilde M\cong M/MI^n.$
\end{Lemma}
\begin{proof}
The inclusion ${\sf cl}({\sf Ann}_{\LL}\: M)=\bigcap ({\sf
Ann}_\LL\: M + I^i) \supseteq I^{n}$  implies ${\sf
Ann}_\LL\:M+I^n={\sf Ann}_\LL\:M+I^{n+1}$ and hence $$MI^n=M({\sf
Ann}_\LL\:M+I^n)=M({\sf Ann}_\LL\:M+I^{n+1})=MI^{n+1}.$$ Using
\eqref{eq_incl_cl_ann}, we get the inclusion $\tilde I^n\subseteq
{\sf Ann}_{\tilde \LL} \: \tilde M.$ Therefore, we obtain  $\tilde
M=\tilde M/\tilde M\tilde I^n\cong M/MI^n.$ It is easy to see that
the composition of $1\otimes \varphi$ with this isomorphism is the
canonical projection.
\end{proof}

\begin{Proposition}\label{prop_tensor} Let $\varphi:\LL\to \tilde \LL $ be a ring homomorphism satisfying \eqref{eq_conditions_Zar}, and  $M$ and $N$ be $\Lambda$-modules such that  ${\sf cl}({\sf Ann}_\Lambda   M+{\sf Ann}_\Lambda   N)\supseteq I^n$. Then  $(M\otimes_\LL N)I^n=(M\otimes_\LL N)I^{n+1}$ and the obvious morphisms induce isomorphisms
$$\tilde M\otimes_{ \Lambda  } \tilde N\cong \tilde M\otimes_{\tilde \Lambda  } \tilde N\cong
(M\otimes_\Lambda   N)\tilde{\ }\cong (M\otimes_\Lambda   N)/(M\otimes_\Lambda   N)I^n \cong (M/MI^n) \otimes_{\Lambda  } (N/NI^n).$$
\end{Proposition}
\begin{proof} Endow $\tilde M\otimes_\Lambda   \tilde N$ by the structure of a $\tilde{\Lambda  }$-module as follows: $(m\otimes n)a=m\otimes (na)$ for $a\in \tilde \Lambda  ,$ $m\in \tilde M,$ $n\in \tilde N.$ Then the induced action of $\Lambda  $ on $\tilde M\otimes_\Lambda   \tilde N$ coincides with the standard action. Hence, ${\sf Ann}_\Lambda   (\tilde M\otimes_\Lambda   \tilde N)=\varphi^{-1}({\sf Ann}_{\tilde \Lambda  } (\tilde M\otimes_\Lambda   \tilde N)).$ Since all ideals in  $\tilde \Lambda  $ are closed and $\varphi:\Lambda  \to \tilde \Lambda  $
is continuous, ${\sf Ann}_\Lambda  (\tilde M\otimes_\Lambda   \tilde N)$ is a closed ideal. Thus, ${\sf Ann}_\Lambda  (\tilde M\otimes_\Lambda   \tilde N)\supseteq I^n$ and ${\sf Ann}_{\tilde \Lambda  }(\tilde M\otimes_\Lambda   \tilde N)\supseteq \tilde{I^n}.$

Now we prove that for any  $b\in \tilde{I^n}$ and  $m\in \tilde M,n\in \tilde N$ the elements $ m\otimes nb$ and $mb\otimes  n$ are equal to zero in $\tilde M\otimes_\Lambda   \tilde N$. The first equality is obvious because
$b \in {\sf Ann}_{\tilde \Lambda  }(\tilde M\otimes_\Lambda   \tilde N)$
and $ m\otimes nb= (m\otimes n)b=0.$
The ring $\Lambda  $ is Noetherian and hence the ideal
$I^n$ is finitely generated  $I^n=(\lambda_1, \dots, \lambda_s)$. Then  $\tilde{I^n}=(\varphi(\lambda_1),\dots,\varphi(\lambda_s)).$
Consider  $b=\sum_{i=1}^s a_i\varphi(\lambda_i) \in \tilde{I^n},$ where $a_i\in \tilde{\Lambda  }.$ Since $\varphi(\lambda_i)$ annihilates $\tilde M\otimes_\Lambda   \tilde N$, we get  $$mb\otimes n=\sum_{i=1}^s m a_i\varphi(\lambda_i)\otimes n=\sum_{i=1}^s (m a_i\otimes n)\varphi(\lambda_i)=0.$$ Therefore, we have that the image of $(\tilde M\otimes_\Lambda   \tilde N\tilde I^n )\oplus (\tilde M\tilde I^n\otimes_\Lambda   \tilde N)$ in $\tilde M \otimes_\Lambda   \tilde N$ vanishes. It follows that  $\tilde M\otimes_\Lambda   \tilde N\cong (\tilde M/\tilde M\tilde I^n)\otimes_\Lambda   (\tilde N/\tilde N\tilde I^n).$ Using \eqref{eq_coinvariants}, we get the isomorphism $\tilde M\otimes_\Lambda   \tilde N\cong ( M/ M{I^n})\otimes_\Lambda   ( N/ N{I^n}).$

The isomorphisms $$\tilde M\otimes_{\tilde \Lambda  } \tilde
N\cong (M\otimes_\Lambda   N)\tilde{\ }\cong (M\otimes_\Lambda
N)/(M\otimes_\Lambda   N)I^n$$ and the equality $(M\otimes_\LL
N)I^n=(M\otimes_\LL N)I^{n+1}$ follow from lemma \ref{lemma_ses}
and the inclusion $${\sf cl}({\sf Ann}_\LL (M\otimes_\LL
N))\supseteq {\sf cl}({\sf Ann}_\LL \: M+ {\sf Ann}_\LL\:
N)\supseteq I^n.$$ Then we only need to prove the isomorphism
$\tilde M\otimes_\Lambda   \tilde N\cong \tilde
M\otimes_{\tilde{\Lambda  }}\tilde N.$ It is sufficient to prove
that for any  $a \in \tilde{\Lambda  }$ and $m\in \tilde M, n\in
\tilde N$ the equality $ma \otimes n= m\otimes na$ holds in
$\tilde M\otimes_\Lambda   \tilde N.$ Since  $\tilde \Lambda
=\varphi(\Lambda  )+\tilde{I}^n,$ $a$ can be presented as follows
$a=\varphi(\lambda)+b,$ where $\lambda\in \Lambda  $ and $b\in
\tilde{I}^n.$ Then the equality $ma \otimes n= m\otimes na$
follows from the equalities $mb\otimes n=0,$ $m\otimes nb=0$ and
$m\lambda\otimes n=m\otimes n\lambda.$
\end{proof}

Let $\sigma:\Lambda  \to \Lambda  $ be an automorphism such that
$\sigma(I)=I$ and $\sigma^2={\sf id}.$ In particular, $\sigma$ is
continuous in the $I$-adic topology.   For a $\Lambda  $-module
$M$ we denote by $M_\sigma$ the $\Lambda  $-module with the same
underlying abelian group and the following action
$m*\lambda=m\sigma(\lambda).$ Define the {\it twisted exterior
square} $\wedge^2_\sigma M$ as the quotient $\Lambda $-module
$$\wedge^2_{\sigma} M=\frac{M\otimes_\Lambda   M_\sigma}{\langle \{ m\otimes m \mid m\in M\}\rangle_\Lambda  }, $$
where $\langle X \rangle_\Lambda  $ means the $\Lambda$-submodule
generated by $X.$

\begin{Corollary}\label{cor_wedge1}
Let $M$ be a $\Lambda$-module such that ${\sf cl}({\sf Ann}_\Lambda  M+\sigma({\sf Ann}_\Lambda  M))\supseteq I^n.$ Then  $(\wedge^2_\sigma M)I^n=(\wedge^2_\sigma M)I^{n+1}$ and the obvious morphisms induce isomorphisms
 $$\wedge^2_\sigma \tilde{M}\cong(\wedge^2_\sigma M)\tilde{\ }\cong (\wedge^2_\sigma M)/(\wedge^2_\sigma M)I^n \cong  \wedge^2_\sigma (M/MI^n).$$
\end{Corollary}
\begin{proof}  For an $\Lambda  $-module $N$ we set $$D(N):=\langle\{ n\otimes n \mid n\in N\} \rangle_\Lambda   \leq N\otimes_\Lambda
N_\sigma.$$ It is sufficient to prove the isomorphisms $$ \tilde
M\otimes_\Lambda   \tilde M_\sigma\cong (M\otimes_\Lambda
M_\sigma)\tilde{\ }\cong (M/MI^n)\otimes_\Lambda (M/MI^n)_\sigma$$
induce  isomorphisms $ D(\tilde M)\cong D(M)\tilde{\ }\cong
D(M/MI^n).$ First we prove that the isomorphism $$\tilde
M\otimes_\Lambda   \tilde  M_\sigma\cong (M/MI^n)\otimes_\Lambda
(M/MI^n)_\sigma$$ induces the isomorphism $D(\tilde M)\cong D(
M/MI^n).$ It is easy to see that the functor $D$ takes
epimorphisms to epimorphisms. Hence the epimorphism $\tilde M\epi
M/MI^n$ induces the epimorphism $D(\tilde M)\to D(M/MI^n).$ From
the other hand $D(\tilde M)\to D(M/MI^n)$ is a monomorphism
because $$\tilde M\otimes_\Lambda   \tilde M_\sigma\to
(M/MI^n)\otimes_\Lambda   (M/MI^n)_\sigma$$  is an isomorphism.

Then we only need to prove that the isomorphism $(M\otimes_\Lambda
M_\sigma)\tilde{\ }\cong \tilde M\otimes_\Lambda  \tilde M_\sigma$
induces the isomorphism $D(M)\tilde{\ }\cong D(\tilde M).$ The
image of $D(M)\tilde{\ }$ in $\tilde M\otimes_\Lambda   \tilde
M_\sigma$ is generated by $\varphi_M(m)\otimes \varphi_M(m)$ for
$m\in M$ and hence $D(\tilde M)$ includes the image. We need to
prove that for any $x\in \tilde M$ the element $x\otimes x$ lies
in the image of $D(M)\tilde{ \ }.$ Since  $\tilde
M=\varphi_M(M)+\tilde M\tilde I^n,$ we get $x=\varphi_M(m)+y,$ for
some $m\in M ,y\in \tilde M\tilde I^n.$ Isomorphisms $$\tilde
M\otimes_\Lambda   \tilde M_\sigma\cong (M/MI^n)\otimes_\Lambda
(M/MI^n)_\sigma\cong (\tilde M/\tilde M\tilde I^n)\otimes_\Lambda
(\tilde M/\tilde M\tilde I^n)_\sigma$$ imply that the elements
$y\otimes\varphi_M(m),$ $\varphi_M(m)\otimes y,$ $y\otimes y$
vanish in $\tilde M\otimes_\Lambda   \tilde M_\sigma$ and hence
$x\otimes x = \varphi_M(m)\otimes \varphi_M(m)$ lies in the image
of $D(M)\tilde{\ }.$
\end{proof}

\begin{Corollary}\label{cor_wedge2}
Let $\LL$ be a commutative Noetherian ring, $I$ be an ideal of $\LL$ and $M$ be a finitely generated $\LL$-module such that ${\sf cl}({\sf Ann}_\Lambda  M+\sigma({\sf Ann}_\Lambda  M))\supseteq I^n.$ Then there are isomorphisms
 $$\wedge^2_\sigma\: \hat{M}\cong \wedge^2_\sigma\: M^\ell \cong (\wedge^2_\sigma M)/(\wedge^2_\sigma M)I^n \cong \wedge^2_\sigma (M/MI^n),$$
 where $\hat M$ is the $I$-adic completion and $M^\ell$ is the localization $M^\ell=M[(1+I)^{-1}].$
\end{Corollary}
\begin{proof}
It follows from corollary \ref{cor_wedge1}, the isomorphisms $\hat M\cong M\otimes_\LL \hat\LL$,
$M^\ell\cong M\otimes_\LL \LL^\ell$ and the fact that the morphisms $\LL\to \hat \LL$ and $\LL\to \LL^{\ell}$ satisfy \eqref{eq_conditions_Zar}.
\end{proof}

\section{$R$-completion of a metabelian group.}

Let $G$ be a  metabelian group and \begin{equation}\label{eq_MGA}
 0\longrightarrow M\longrightarrow G\overset{\pi}\longrightarrow A \longrightarrow 1
\end{equation} is a short exact sequence of groups, where $A$ is an abelian group and $M$ is a right $A$-module with the action defined
by conjugation. We assume that $M=\ker(\pi)\subseteq G.$ We use
the multiplicative notation for $A$ and $G$ but for $M$ we use
both the multiplicative and the additive notations. We use $*$ for
the action of $\mathbb Z[G]$ and $\mathbb Z[A]$ on $M$  in order
to separate it from the multiplication in the group. Therefore,
for $m,m_1,m_2\in M,$ $g,g_1,\dots, g_l\in G$ and $k_1,\dots,
k_l\in \mathbb Z$ we have $(m^{k_1})^{g_1}\dots
(m^{k_l})^{g_l}=m*(\sum_{i=1}^lk_ig_i)=m*\pi(\sum_{i=1}^lk_ig_i)$
and $m_1^gm_2^g=(m_1+m_2)*g=(m_1+m_2)*\pi(g).$

There is a notion of $R$-completion of a group for subrings of
$\mathbb Q$ and for $R=\mathbb Z/n$ (see \cite{Bousfield-Kan}).
All subrings of $\mathbb Q$ have the form $R=\mathbb Z[J^{-1}],$
where $J$ is a set of prime numbers. We are going to describe the
$R$-completion $\hat G_R$ of the metabelian group $G$ in terms of
$A$ and $M$ in these two cases separately. For the case $R=\mathbb
Z[J^{-1}]$ we need some information about Malcev $R$-completion.

\subsection{Malcev $R$-completion.}

In this subsection by $R$ we denote the ring $\mathbb Z[J^{-1}]\subseteq \mathbb Q.$  Recall the notion of Malcev $R$-completion \cite{Hilton}, \cite{Hilton-Mislin-Roitberg}, \cite{Quillen_rational_homotopy_theory}.  A group $G$ is said to be {\it $J$-local} or {\it uniquely $J$-divisible} if the map $g\mapsto g^p$ is a bijection for $p\in J.$ The embedding of the category of all $J$-local nilpotent groups to the category of all nilpotent groups ${\sf Nil}_J \hookrightarrow {\sf Nil}$ has the left adjoint functor  called Malcev $R$-completion
$$-\otimes R: {\sf Nil}\to {\sf Nil}_J.$$
 Thus, if $H$ is a nilpotent group and  $H'$ is $J$-local nilpotent group, there is a natural isomorphism ${\sf Hom}(H,H')\cong {\sf Hom}(H\otimes R, H'),$
 and the unit of the adjunction $\eta_H:H\to H\otimes R$ is the universal homomorphism from $H$ to a $J$-local nilpotent group.
If $H$ is abelian, then $H\otimes  R$ is the ordinary tensor product. The functor $-\otimes  R$ preserves short exact sequences.

Let $K$ denote an Artinian quotient ring of $R.$ In other words,
$$K=\left\{\begin{array}{ll}
\mathbb Q, & \text{ if } R=\mathbb Q\\
\mathbb Z/n, & \text{ if } R=\mathbb Z[J^{-1}]\ne \mathbb Q
\end{array}\right.,$$
 where $n$ is a natural number such that the prime divisors do not lie in $J.$

\begin{Lemma}\label{lemma_lemma_finite_module_nilpotent} Let $H'\mono H\epi H''$ be a short exact sequence of finitely generated nilpotent groups.
If $N$ is a nilpotent $K[H]$-module finitely generated over $K$,
then the homology group $H_i(H',N)$ is a nilpotent $K[H'']$-module
finitely generated over $K.$
\end{Lemma}
\begin{proof}
Since $H',H,H''$ are  finitely generated nilpotent groups, then
the group rings $K[H'],K[H], K[H'']$ are Noetherian
\cite{Brown-Dror}. It follows that there exists a free resolution
$P_\bullet$  of the trivial $K[H']$-module $K$ that consists of
finitely generated free $K[H']$-modules. Since $N$ is a nilpotent
$K[H]$-module finitely generated over $K$, the module
$P_i\otimes_{KH'}N$ has this property too. The homology group
$H_i(H',N)$ is a subquotient of $P_i\otimes_{K[H']}N,$ and hence
it has this property too.
\end{proof}
\begin{Proposition}\label{prop_finite_module}
Let $H$ be a nilpotent finitely generated group and  $N$ be a
nilpotent  $K[H\otimes R]$-module that is finitely generated over
$K$, then the homomorphism $H\to H\otimes R$ induces the
isomorphism
$$H_*(H,N)\cong H_*(H\otimes R,N).$$
\end{Proposition}
\begin{proof} Let $R=\mathbb Z[J^{-1}]\ne \mathbb Q,$ and hence $K=\mathbb Z/n.$
First we prove the proposition for an abelian group $H$ and $N=\mathbb Z/p,$ where $p$ is a prime divisor of $n.$
Homology $H_*(X,\mathbb Z/p)$ of an abelian group $X$ is isomorphic to ${\bigwedge}^*(X/p)\otimes  \Gamma_*({}_pX)$
(see \cite[V, 6.6]{Brown}). Then we only need to note that $H/p\cong (H\otimes R)/p$ and ${}_pH\cong {}_p(H\otimes R).$

Let now $H$ is abelian and $N$ is a finite nilpotent
$K[H]$-module. All  nilpotent $K[H\otimes R]$-modules finitely
generated over $\mathbb Z/n$ can be obtained by a sequence of
extensions from the trivial modules $\mathbb Z/p,$ where $p$ is a
prime divisor of $n.$ Then we need to prove that the class of
$K[H\otimes R]$-modules with the property $H_*(H,N)\cong
H_*(H\otimes R,N)$ is closed under extensions. It follows easily
from the homology long exact sequence and the five lemma.

Prove the general case. We need to prove that the class of groups with this property is closed under extensions. Let $H'\mono H\epi H''$ is a short exact sequence of finitely generated nilpotent groups such that the proposition holds for $H'$ and $H''$. We prove it for $H.$ Let $N$ be a  nilpotent $KH$-module finitely generated over $K$.  Consider the morphism of the short exact sequences
$$
\begin{tikzpicture}
  \matrix (m)
 [matrix of math nodes,row sep=1cm,column sep=1cm,minimum width=2em]
{
 1 & H' & H & H'' & 1    \\
 1 & H'\otimes R & H\otimes R & H''\otimes R & 1.    \\
   };
 \path[->]
(m-1-1) edge  (m-1-2)
(m-1-2) edge  (m-1-3)
(m-1-3) edge  (m-1-4)
(m-1-4) edge  (m-1-5)
(m-2-1) edge  (m-2-2)
(m-2-2) edge  (m-2-3)
(m-2-3) edge  (m-2-4)
(m-2-4) edge  (m-2-5)
(m-1-2) edge (m-2-2)
(m-1-3) edge (m-2-3)
(m-1-4) edge (m-2-4)
;
\end{tikzpicture}
$$
It induces the the morphism of the corresponding
Lyndon-Hochschild-Serre spectral sequences $E\to E^{R}$. It is
sufficient to prove that the morphism $E\to E^R$ is an
isomorphism. By induction hypothesis we know $H_q(H',N)\cong
H_q(H'\otimes R,N).$ By Lemma
\ref{lemma_lemma_finite_module_nilpotent} the $KH''$-module
$H_q(H',N)$ is  finite and nilpotent. Then again by induction
hypothesis we have $H_p(H'',H_q(H',N))\cong H_p(H''\otimes
R,H_q(H'\otimes R,N)).$ It follows that the morphism $E\to E^R$ is
an isomorphism.

The case of $R=K=\mathbb Q$ can be proved similarly, using the formula $H_*(X,\mathbb Q)\cong \bigwedge^*(X\otimes \mathbb Q)$ for an abelian group $X.$
\end{proof}

\subsection{$\mathbb Z[J^{-1}]$-completion of a metabelian group.}
In this section we assume $R=\mathbb Z[J^{-1}].$ For $\alpha \in \mathbb R$ we denote $\binom{\alpha}{n}:=\alpha(\alpha-1)\dots (\alpha-n+1)/n!.$

\begin{Lemma}
If $\alpha\in \mathbb Z[J^{-1}]$ then  $\binom{\alpha}{n}\in \mathbb Z[J^{-1}]$.
\end{Lemma}
\begin{proof}
Let $q\notin J$ be a prime number, $v$  be the $q$-adic value of
$n!$ and $l\in \mathbb Z$ such that $l\alpha^{-1}\in \mathbb Z,$ \
$l\cdot \alpha^{-1} \equiv 1\ ({\sf mod}\ q^v)$ and $l\geq n.$
Consider the epimorphism $\mathbb Z[J^{-1}]\epi \mathbb Z/ q^v.$
The image of $\alpha$ coincides with the image of $l,$ and hence
the image of $\alpha(\alpha-1)\dots (\alpha-n+1)$ coincides with
the image of $l(l-1)\dots (l-n+1).$ Since $l(l-1)\dots (l-n+1)$ is
divisible by $n!,$ we obtain that the image of
$\alpha(\alpha-1)\dots (\alpha-n+1)$ vanishes, and hence the
$q$-adic value of $\alpha(\alpha-1)\dots (\alpha-n+1)$ is greater
than or equal to $v.$ It follows that the $q$-adic value of
$\binom{\alpha}{n}$ is non-negative for any $q\notin J.$ Hence
$\binom{\alpha}{n}\in \mathbb Z[J^{-1}].$
\end{proof}

Let $\LL$ be a complete $\mathbb Z[J^{-1}]$-algebra with respect to an ideal $\mathfrak{a}$ i.e. $\varphi:\LL\to \hat \LL_{\mathfrak{a}}$ is an isomorphism. For $x \in 1+\mathfrak{a}$ we denote

\begin{equation}\label{eq_alpha_power}
x^{[\alpha]}=\sum_{n=0}^{\infty}  \binom{\alpha}{n} \: (x-1)^n.
\end{equation}

\begin{Lemma}\label{lemma_J-local}
For any $\alpha,\beta\in \mathbb Z[J^{-1}]$,  $n\in \mathbb N$ and $x\in 1+\mathfrak{a}$ the following equalities hold
$$(x^{[\alpha]})^{[\beta]}=x^{[\alpha\beta]},\hspace{1cm} x^{[\alpha]}x^{[\beta]}=x^{[\alpha+\beta]}, \hspace{1cm} x^{n}=x^{[n]}. $$ In particular, $(x^p)^{[1/p]}=x=(x^{[1/p]})^{p},$ and $x\cdot x^{[-1]}=1=x^{[-1]}\cdot x.$ Therefore, $1+\mathfrak{a}$ is a $J$-local group. Moreover, $\gamma_n(1+\mathfrak{a})\subseteq 1+\mathfrak{a}^n,$ where $\{\gamma_n(1+\aa)\}$ is the lower central series of the group $1+\aa$.
\end{Lemma}
\begin{proof} The equality $x^n=x^{[n]}$ follows from the binomial theorem.
From the standard course of mathematical analysis we know that for $t>0$ and $\alpha\in \mathbb R$ the equality $t^{\alpha}=\sum_{n=0}^{\infty} \binom{\alpha}{n} (t-1)^n$ holds. Denote by $\mathfrak{b}$ the ideal of $R[t]$  generated by the element  $(t-1).$  Then for $\alpha,\beta\in R$ the equalities $(t^{\alpha})^{\beta}=t^{\alpha\beta}$ and
 $t^{\alpha +\beta}=t^{\alpha}t^{\beta}$
  imply that in the algebra $\widehat{R[t]}_{\mathfrak{b}}$ of 'power series of $(t-1)$'  the
   corresponding identities hold $(t^{[\alpha]})^{[\beta]}=t^{[\alpha\beta]}, t^{[\alpha+\beta]}=t^{[\alpha]}t^{[\beta]}.$ Consider the homomorphism $\xi:R[t]\to \LL,$ that takes $t$ to $x.$ Endow $R[t]$ with the $\mathfrak{b}$-adic topology. Then $\xi$
   is continuous, and hence it induces a continuous homomorphism $\hat \xi:\widehat{R[t]}_{\mathfrak{b}} \to \LL.$ Since $\hat \xi$ is continuous, we get $\hat \xi(f^{[\alpha]})=\hat \xi(f)^{[\alpha]}$ for all $f\in 1+\hat{\mathfrak{b}}.$ Therefore, the equalities $(t^{[\alpha]})^{[\beta]}=t^{[\alpha\beta]}, t^{[\alpha+\beta]}=t^{[\alpha]}t^{[\beta]}$ imply  $(x^{[\alpha]})^{[\beta]}=x^{[\alpha\beta]}, x^{[\alpha+\beta]}=x^{[\alpha]}x^{[\beta]}.$

Prove the inclusion $\gamma_n(1+\mathfrak{a})\subseteq 1+\mathfrak{a}^n$ by induction on $n.$ For $n=1$ it is obvious. Assume that  $\gamma_{n-1}(1+\mathfrak{a})\subseteq 1+\mathfrak{a}^{n-1}.$ For two elements $a,b \in 1+\mathfrak{a}$ we write $a\equiv b$ if their images in the quotient group  $(1+\mathfrak{a})/(1+\mathfrak{a}^{n})$ are equal. Then for $y\in \mathfrak{a}$ and $z\in \mathfrak{a}^n$ we have $1+y+z\equiv 1+y$ because
$(1+y)^{-1}(1+y+z)=1+(1+y)^{-1}z$ and $(1+y)^{-1}z\in \mathfrak{a}^n.$ Chose $1+x\in \gamma_{n-1}(G)$ and $1+y\in 1+\mathfrak{a}.$ Then $x\in \mathfrak{a}^{n-1}, y\in \mathfrak{a}$ and we have
$$[1+x,1+y]=(1+x)^{-1}(1+y)^{-1}(1+x)(1+y)= $$
$$= \left(\sum_{i=0}^{\infty} (-1)^ix^i \right)\left( \sum_{i=0}^{\infty} (-1)^iy^i \right)(1+x)(1+y) \equiv (1-x)(1-y)(1+x)(1+y)\equiv $$
$$\equiv 1-x-y+x+y =1.$$
Therefore $[1+x,1+y]\in 1+\mathfrak{a}^n,$ and hence $\gamma_n(1+\mathfrak{a})\subseteq 1+\mathfrak{a}^n.$

\end{proof}

\begin{Corollary}\label{cor_theta}
Let $A$ be an abelian group and $\widehat{R[A]}$ be the $I_R$-adic
completion of the group algebra, where $I_R=I_R(A)$ is the
augmentation ideal. Then the homomorphism $\tilde \theta:A\otimes
R \to 1+\hat I_R$ defined by the formula $\tilde \theta(a\otimes
\alpha)= a^{[\alpha]}$ is the unique homomorphism such that the
following diagram is commutative
\begin{equation}\label{eq_theta}
\begin{tikzpicture}
  \matrix (m)
 [matrix of math nodes,row sep=1cm,column sep=0.5cm,minimum width=2em]
{
   & A &    \\
  A\otimes R & & 1+\hat I_R. \\
   };
 \path[->]
(m-1-2) edge node [left] {$\eta$} (m-2-1)
(m-1-2) edge node [right] {$\iota$} (m-2-3)
(m-2-1) edge node [above] {$\tilde \theta$} (m-2-3)
;
\end{tikzpicture}
\end{equation}
\end{Corollary}

Let $G$ be a metabelian group with a metabelian decomposition
$M\mono G \epi A$. $G$ is nilpotent if and only if $M$ is a
nilpotent $\mathbb Z[A]$-module (i.e. $M*I^m=0$ for $m>\!>0$). In
this case  $(M\otimes R)^{\wedge}_{I_R}=M\otimes R$, and hence
$M\otimes R$ has the natural structure of $\widehat{R[A]}$-module.
The composition of homomorphisms $A\otimes R \overset{\tilde
\theta}\longrightarrow 1+\hat I_R \hookrightarrow \widehat{R[A]}$
we denote by
\begin{equation}\label{eq_theta_1}
\theta_R: A\otimes R \longrightarrow \widehat{R[A]}.
\end{equation}
Observe that $\theta_R(x^{-1})=\hat \sigma(\theta_R(x)),$ where
$\sigma:{R[A]}\to {R[A]}$ is the antipode.

\begin{Lemma}\label{lemma_action} Let $G$ be a  nilpotent metabelian group with a metabelian decomposition $M\mono G\epi A$.
Then the group $G\otimes R$ is a metabelian nilpotent group with
the metabelian decomposition $M\otimes R \mono G\otimes R\epi
A\otimes R,$ and the action of $A\otimes R$ on $M\otimes R$ is
induced by the structure of $\widehat{R[A]}$-module via the
homomorphism $\theta_R: A\otimes R\to \widehat{R[A]}.$
\end{Lemma}
\begin{proof}
Since $-\otimes R$ takes short exact sequences to short exact
sequences, and on abelian groups it is the usual tensor product,
we get the metabelian decomposition $M\otimes R \mono G\otimes R
\epi A\otimes R.$ Using the morphism of sequences

$$
\begin{tikzpicture}
  \matrix (m)
 [matrix of math nodes,row sep=1cm,column sep=1.7cm,minimum width=2em]
{
 0 &  M & G & A  & 1 \\
0 &  M\otimes R & G\otimes R & A\otimes R & 1. \\
   };
 \path[->]
(m-1-1) edge (m-1-2)
(m-1-2) edge (m-1-3)
(m-1-3) edge node [above] {$\pi$} (m-1-4)
(m-1-4) edge (m-1-5)
(m-2-1) edge (m-2-2)
(m-2-2) edge (m-2-3)
(m-2-3) edge node [above] {$\pi\otimes 1$} (m-2-4)
(m-2-4) edge (m-2-5)
(m-2-2) edge  (m-2-3)
(m-1-2) edge node [left] {$\eta_M$} (m-2-2)
(m-1-3) edge node [left] {$\eta_G$} (m-2-3)
(m-1-4) edge node [left] {$\eta_A$} (m-2-4)
;
\end{tikzpicture}
$$
we get the identity $(m\otimes 1)*(a\otimes 1)=(m*a) \otimes 1.$
An endomorphism of the abelian group $M$ lifts uniquely to an
endomorphism of the abelian group $M\otimes R.$ Hence, $(m\otimes
r)*(a \otimes 1)=(m*a)\otimes r$ for all $m\in M,a\in A, r\in R.$
Therefore, the action of $A\otimes R$ on $M\otimes R$ extends the
induced action of $A.$

Consider the module $M\otimes R$ as a filtered $R$-module, where
the filtration is given by $\mathcal{F}_i=[M\otimes
R,\gamma_i(G\otimes R)]$ where $\{\gamma_i(G\otimes R)\}$ is the
lower central series of $G\otimes R.$ Since $G\otimes R$ is
nilpotent, the filtration is finite. Observe that $\mathcal F_i$
is a normal subgroup of $G\otimes R$ and equivalently a
$\widehat{R[A]}$-submodule of $M\otimes R$. Consider the
$R$-algebra $\LL={\sf End}_{R-{\sf Filt}}(M\otimes R)$ of
$R$-endomorphisms preserving filtration, and the ideal
$\mathfrak{a}=\{f\in \LL\mid f(\mathcal F_i)\subseteq  \mathcal
F_{i+1}\}.$ Since $\mathfrak{a}^m=0$ for $m>\!>0,$ the $R$-algebra
$\LL$ is complete with respect to the $\mathfrak{a}$-adic
topology. By Lemma \ref{lemma_J-local},  $1+\mathfrak{a}$ is a
$J$-local nilpotent group. Since $M$ is a $R[A]$-module and
$\{\mathcal F_i\}$ are submodules, we obtain the algebra
homomorphism $R[A] \to \LL.$ The ideal $I_R$ is generated by
elements $(a-1)$ for $a\in A.$ Let $g\in G$ such that $\pi(g)=a.$
Thus, $x*(a-1)=-x+x*a=x^{-1}\cdot x^g=[x,g]$ for $x\in M\otimes
R,$ and hence, the image of $I_R$ lies in $\mathfrak{a}.$ Then the
obtain the continuous homomorphism $\widehat{R[A]}\to \LL$ that
induces the group homomorphism $\sigma: 1+\hat I_R \to
1+\mathfrak{a}.$ Similarly, the action of $A\otimes R$ on
$M\otimes R$ induces a homomorphism $\tau: A\otimes R\to
1+\mathfrak{a}.$ Then we obtain the following diagram.
$$
\begin{tikzpicture}
  \matrix (m)
 [matrix of math nodes,row sep=1.5cm,column sep=2cm,minimum width=2em]
{
  A & 1+\hat I_R     \\
  A\otimes R &  1+\mathfrak{a} \\
   };
 \path[->]
(m-1-1) edge node [above] {$\iota$} (m-1-2)
(m-1-1) edge node [left] {$\eta$} (m-2-1)
(m-2-1) edge node [above] {$\tau$} (m-2-2)
(m-1-2) edge node [right] {$\sigma$} (m-2-2)
(m-2-1) edge node [above] {$\tilde \theta$} (m-1-2)
;
\end{tikzpicture}
$$
By Corollary \ref{cor_theta} we have $\tilde \theta \eta=\iota.$
Since the action of $A\otimes R$ extends the action of $A,$ we
have $\tau\eta=\sigma\iota.$ Thus we have $\tau\eta=\sigma\tilde
\theta\eta,$ and using the universal property of $\eta,$ we get
$\tau=\sigma\tilde \theta.$ Therefore, the action of $A\otimes R$
on $M\otimes R$ is induced by the structure of
$\widehat{R[A]}$-module via the homomorphism $\theta_R.$
\end{proof}

The $R${\it -completion} of a group $G$ is defined as follows
$$\hat G_{ R}:=\ilimit  (G/\gamma_i(G))\otimes  R.$$

\begin{Proposition}\label{proposition_completion_ZJ}
Let $R$ be the ring $\mathbb Z[J^{-1}]$, $G$ be a  metabelian group as in \eqref{eq_MGA}.
 Denote by $(M\otimes R)^{\wedge}_{I_R}$ the $I_R$-adic completion of the $R[A]$-module $M\otimes R.$ Then there is a short exact sequence
$$0 \longrightarrow(M\otimes R)^{\wedge}_{I_R}\longrightarrow \hat G_R \longrightarrow A\otimes R \longrightarrow 1,$$ where the action by conjugation of
$A\otimes R$ on $(M\otimes R)^{\wedge}_{I_R}$ coincides with the
action that induced by the structure of $\widehat{R[A]}$-module
via the homomorphism $\theta_R: A\otimes R\to \widehat{R[A]}.$
\end{Proposition}
\begin{proof}
 For $g\in G$ and $m\in M$ we have $m*(g-1)=-m+m*g=m^{-1}m^g=[m,g].$ Thus, we have
$M*I^i=[M,G,\dots,G]$ and hence
\begin{equation}\label{eq_inclusion_of_towers_ZJ}
M*I^i\subseteq \gamma_i(G)\subseteq M*I^{i-1}.
\end{equation}
Therefore, we obtain an isomorphism $\hat{G}_R=\ilimit
(G/\gamma_i(G))\otimes R\cong \ilimit (G/(M*I^i))\otimes R.$ By
Lemma \ref{lemma_action} the short exact sequences $M/(M*I^i)\mono
G/(M*I^i) \epi A$ give the short exact sequences
$(M/(M*I^i))\otimes R \mono (G/(M*I^i))\otimes R \epi A\otimes R,$
and the action on $(M/(M*I^i))\otimes R$ is induced by the
structure of $\widehat{R[A]}$-module via the homomorphism $\theta:
A\otimes R\to \widehat{R[A]}.$  The inverse sequence
$(M/(M*I^i))\otimes R$ satisfies the Mittag--Leffler condition,
and hence we have the following short exact sequence:
$$0\longrightarrow \ilimit (M/(M*I^i))\otimes R \longrightarrow \hat{G}_R \longrightarrow A\otimes R \longrightarrow 1.$$
Since $R$ is a flat $\mathbb Z$-module, we have
$(M/(M*I^i))\otimes R\cong (M\otimes R)/((M\otimes R)*I_R^i)$.
Therefore, we get $$\ilimit (M/(M*I^i))\otimes R\cong (M\otimes
R)^{\wedge}_{I_R}.$$ The action of $A\otimes R$ on $(M\otimes
R)^{\wedge}_R$ is induced by $\theta,$ because the action of
$A\otimes R$ on the quotients $$(M/(M*I^i))\otimes R \cong
(M\otimes R)/((M\otimes R)*I_R^i)$$ is induced by $\theta$.
\end{proof}

\subsection{$\mathbb Z/n$-completion of a metabelian group.}

In this subsection we denote by $R$ the ring $\mathbb Z/n$.
The {\it $n$-lower central series} of a group $G$ is defined as follows: $\gamma^{[n]}_1(G)=G$ and $\gamma^{[n]}_{i+1}(G)={\rm ker}\big(\gamma^{[n]}_i(G)\to
(\gamma_i^{[n]}(G)/[G,\gamma_i^{[n]}(G)])\otimes \mathbb Z/n\big).$
Then $R$-completion of a group $G$ is
$$\hat G_R=\ilimit G/\gamma_i^{[n]}(G).$$
If $p$ is a prime number and $G$ is a finitely generated group,
the $\mathbb Z/p$-completion $\hat G_{\mathbb Z/p}$ coincides with
the $p$-profinite completion $\hat G_p$
\cite[IV,2.3]{Bousfield-Kan}
\begin{equation}\label{eq_p-profinite}
\hat G_{\mathbb Z/p}\cong \hat G_p.
\end{equation}
Further we will assume that $A$ is a finitely generated abelian
group. For a normal subgroup $H$ of $G$ we denote by ${\sf
p}_n(H)$  the normal subgroup generated by the set of powers
$\{h^n\mid h\in H\}$ and we set $\mathcal{P}_n(H)={\sf
p}_n(H)\cdot [H,G].$ It is easy to verify that
$$\gamma^{[n]}_{i+1}(G)=\mathcal{P}_n(\gamma_i^{[n]}(G)).$$ By
$I_n$ we denote the kernel of the composition of the augmentation
map and the canonical projection $$I_n=\ker(\mathbb
Z[A]\longrightarrow  \mathbb Z/n).$$ Equivalently we can describe
it as follows $I_n=I+(n),$ where $(n)$ is the ideal of $\mathbb
Z[A]$ generated by $n.$ It is easy to see that   $\sum_{i=1}^n
a_i\in I_n,$ for any sequence $a_1,\dots,a_n\in A.$

Since $A$ is finitely generated we can fix a natural number $T$ such that $a^{n^{T+1}}=1$ implies $a^{n^T}=1$ for $a\in A.$ In other words,
\begin{equation}\label{eq_t}
(a\in A^{n^{T}}\ \text{ and } \ a^n=1) \hspace{1cm} \Rightarrow\hspace{1cm} a=1.
\end{equation}

\begin{Lemma}\label{lemma_incl} Let $G$ be a metabelian group as in \eqref{eq_MGA},
$N$ be a submodule of $M$ and $H$ be a normal subgroup of $G$. Then the following holds.
\begin{enumerate}
\item $N*I_n=\mathcal{P}_n(N);$
\item ${\sf p}_n(H\cdot N)\subseteq {\sf p}_n(H)\cdot (N*I_n);$
\item if $H\cap M\subseteq N$ then $[{\sf p}_n(H),G]\subseteq N*I_n;$
\item if $H\cap M\subseteq N$ and $H\subseteq \pi^{-1}(A^{n^T})$ then ${\sf p}_n({\sf p}_n(H)) \cap M\subseteq N*I_n;$
\item if $H\cap M\subseteq N$ and $H\subseteq \pi^{-1}(A^{n^T})$  then $\mathcal{P}_n(\mathcal{P}_n(H)) \cap M\subseteq N*I_n;$
\end{enumerate}
\end{Lemma}
\begin{proof}
(1) It follows from the equalities $N*I=[N,G]$ and $N*(n)={\sf p}_n(N)$.

(2) It follows from the equality $(hm)^n=h^n \cdot
(\prod_{i=0}^{n-1} m^{h^{i}})=h^n \cdot (m* (\sum_{i=0}^{n-1}
h^i))$ and the fact that $\pi(\sum_{i=0}^{n-1} h^i)\in I_n$.

(3) Using the equalities $[x_1 x_2,y]=[x_1,y]^{x_2}\cdot [x_2,y]$
and $[x^z,y]=[x,y^{z^{-1}}]^{z^{-1}}$ we get that the normal
subgroup $[{\sf p}_n(H),G]$ is generated by the elements of the
form $[h^n,g]$ as a normal subgroup, where $h\in H$ and $g\in G.$
Moreover, we have $$[h^n,g]=\prod_{i=0}^{n-1}
[h,g]^{h^i}=[h,g]*(\sum_{i=0}^{n-1} h^i).$$ Using the inclusions
$[G,G]\subseteq M$ and $H\cap M\subseteq N$ we obtain $[h,g]\in N$
and hence $[h^n,g]\in N*I_n.$

(4) Consider an element $x_1^nx_2^n\dots x_l^n \in {\sf p}_n({\sf
p}_n(H))\cap M$ where $x_i\in {\sf p}_n(H).$ Thus
$\pi(x_1^nx_2^n\dots x_l^n)=\pi(x_1\dots x_l)^n=1.$ Since
$H\subseteq \pi^{-1}(A^{n^t})$, using \eqref{eq_t} we obtain
$\pi(x_1\dots x_l)=1$ and hence $x_1\dots x_l\in M.$ Moreover,
using the inclusion $H\cap M\subseteq N$, we obtain $x_1\dots
x_l\in N.$ From the other hand, we have $$x_1^n\dots x_l^n
(x_1\dots x_l)^{-n}\in [{\sf p}_n(H),{\sf p}_n(H)]\subseteq [{\sf
p}_n(H),G] \subseteq N*I_n$$ and $(x_1\dots x_l)^n\in
N*(n)\subseteq N*I_n.$ Therefore, $x_1^n\dots x_l^n\in N*I_n.$

(5) Since $H \cap M\subseteq N$ and $[G,G]\subseteq M,$ we have
$[H,G]\subseteq N$ and hence $[\mathcal{P}_n(H),G]\subseteq N.$
Thus, we obtain $$ \mathcal{P}_n(\mathcal{P}_n(H))\cap M=
\big({\sf p}_n(\mathcal{P}_n(H))\cdot
[\mathcal{P}_n(H),G]\big)\cap M= \big({\sf p}_n(\mathcal{P}_n(H))
\cap M\big)\cdot [\mathcal{P}_n(H),G].$$ Therefore, we need to
prove the inclusions ${\sf p}_n(\mathcal{P}_n(H))\cap M\subseteq
N*I_n$ and $[\mathcal{P}_n(H),G]\subseteq N*I_n.$ Using the
inclusion $[H,G]\subseteq N$, the inclusion (2), the inclusion
$N*I_n\subseteq M,$ and (4) we obtain $${\sf
p}_n(\mathcal{P}_n(H))\cap M={\sf p}_n({\sf p}_n(H)\cdot
[H,G])\cap M\subseteq {\sf p}_n({\sf p}_n(H)\cdot N)\cap M
\subseteq$$ $$\subseteq \big({\sf p}_n({\sf p}_n(H))\cdot
(N*I_n)\big)\cap M=\big({\sf p}_n({\sf p}_n(H))\cap M\big) \cdot
(N*I_n)\subseteq N*I_n.$$
\end{proof}
\begin{Lemma}\label{lemma_inclusions_n}
The following inclusions hold for $i\geq 0:$
$$M*I_n^i \subseteq \gamma^{[n]}_{i+1}(G)\cap M, \hspace{1cm} \gamma^{[n]}_{2i+T+1}(G)\cap M\subseteq M*I_n^i.$$
\end{Lemma}
\begin{proof}

The proof is by induction on $i.$ The base is obvious for both
cases. Suppose that these inclusions hold for $i=j$ and prove them
for $i=j+i.$ In order to prove the first inclusion we only need to
prove $M*I_n^{j+1}\subseteq \gamma^{[n]}_{j+2}(G).$ Using lemma
\ref{lemma_incl}, we get the first required inclusion
$$M*I_n^{j+1}=\mathcal{P}_n(M*I_n^{j})\subseteq
\mathcal{P}_n(\gamma_j^{[n]}(G))=\gamma^{[n]}_{j+1}(G).$$ Since
$\gamma^{[n]}_{2j+T+1}(G)\subseteq \gamma^{[n]}_{T+1}(G)\subseteq
\pi^{-1}(A^{n^T})$ and $\gamma^{[n]}_{2j+T+1}(G)\cap M\subseteq
M*I_n^j$, we can use (5) of lemma \ref{lemma_incl} and obtain the
second required inclusion
$$\gamma^{[n]}_{2j+2+T+1}(G)\cap M=
\mathcal{P}_n(\mathcal{P}_n(\gamma^{[n]}_{2j+T+1}(G)))\cap
M\subseteq (M*I_n^{j})*I_n=M*I_n^{j+1}.$$
\end{proof}
\begin{Corollary}
$\hat M_{I_n}=\ilimit M/(\gamma^{[n]}_i(G)\cap M).$
\end{Corollary}
\begin{Corollary}\label{corollary_a_v_bolshoi_step} For any $i,n\in \mathbb N$ we have the following.
\begin{enumerate}
\item Let $\mathbb Z[t,t^{-1}]$ be the ring of Laurent polynomials. Then $$t^{n^{2i+1}} \equiv 1 \ {\sf mod}\  (n,(t-1))^i.$$

\item Let $a\in A.$ Then $a^{n^{2i+1}}-1\in I_n^i.$

\item Let $\mathfrak{R}$ be an associative ring whose
characteristic divides $n,$ and $r$ be an invertible element of
$\mathfrak{R}.$ Then the element $r^{n^{2i+1}}-1$ is divisible by
$(r-1)^i.$
\end{enumerate}
\end{Corollary}
\begin{proof} It is easy to see that (2) and (3) follow from (1). Prove (1).
Let $A$ be the infinite cyclic group $\langle t \rangle.$  Then
$\mathbb Z[A]=\mathbb Z[t,t^{-1}],$ and  $T=0.$  Consider the
semidirect product $H:=\langle t \rangle \ltimes \mathbb
Z[t,t^{-1}]$ and the short exact sequence $$\mathbb
Z[A]/(\gamma^{[n]}_{2i+1}(H)\cap \mathbb Z[A])\mono
H/\gamma^{[n]}_{2i+T+1}(H) \epi A/A^{n^{2i+T+1}}.$$ Then the
action of $A$ on $\mathbb Z[t,t^{-1}]$ induces an action of
$A/A^{n^{2i+T+1}}$  on $\mathbb Z[A]/(\gamma_i^{[n]}(H)\cap
\mathbb Z[A])$. We have $I_n=(n,t-1).$ By Lemma
\ref{lemma_inclusions_n} we get the action of $A$ on $\mathbb
Z[t,t^{-1}]$ induces an action of $A/A^{n^{2i+1}}$  on $\mathbb
Z[t,t^{-1}]/I_n^i$. Hence $t^{n^{2i+1}}$ acts trivially on
$\mathbb Z[t,t^{-1}]/(n,t-1)^i$ by multiplication.
\end{proof}

Corollary \ref{corollary_a_v_bolshoi_step} implies that the
multiplicative homomorphism $A\to \mathbb Z[A]$ induces a
homomorphism $\theta'_i: A/A^{n^{2i+T+1}} \to \mathbb Z[A]/I_n^i.$
Then applying the inverse limit we get the continuous homomorphism
\begin{equation}
\theta_R:\hat A_n \to \widehat{\mathbb Z[A]}_{I_n}.
\end{equation}
Observe that $\theta_R(x^{-1})=\hat \sigma(\theta_R(x)),$ where
$\sigma:{\mathbb Z}[A]\to {\mathbb Z}[A]$ is the antipode.

\begin{Proposition}
Let $R$ be the ring $\mathbb Z/n$, $G$ be a  metabelian group as in \eqref{eq_MGA}, where $A$ is a finitely generated abelian group.
 Denote by $\hat{M}_{I_n}$ the $I_n$-adic completion of the $\mathbb Z[A]$-module $M.$ Then there is a short exact sequence
$$0 \longrightarrow \hat{M}_{I_n}\longrightarrow \hat G_R \longrightarrow \hat A_n  \longrightarrow 1,$$
whose morphisms are induced by \eqref{eq_MGA}, and the action by
conjugation of $\hat A_n$ on $\hat M_{I_n}$ coincides with the
action that induced by the structure of $\widehat{\mathbb
Z[A]}_{I_n}$-module via the homomorphism $\theta_R:\hat A_n\to
\widehat{\mathbb Z[A]}_{I_n}.$
\end{Proposition}
\begin{proof} Denote $M_i=M/(\gamma^{[n]}_i(G)\cap M)$ and
consider the short exact sequence $M_i \mono G/\gamma^{[n]}_i(G)
\epi A/A^{n^i}.$ The inverse sequence $M_i$ satisfies
Mittag-Leffler condition and hence we get the short exact sequence
$\ilimit M_i \mono \hat G_R \epi \hat A_n.$ Finally, by lemma
\ref{lemma_inclusions_n} we obtain the required isomorphism
$\ilimit M_i \cong \ilimit M/(M*I_n^i)=\hat M_{I_n}.$ Note that
$M_i*I_n^i=0.$ Then the module $M_i$ has the natural structure of
a $\widehat{\mathbb Z[A]}_{I_n}$-module. In order to prove that
the action of $\hat A_n$ on $M$ is induced by $\theta',$ we only
need to prove it for $M_i.$ But it is obvious, because $\hat
A_n/\hat A_n^{n^i}\cong A/ A^{n^i}.$
\end{proof}

Now we give a slight different formulation of Proposition \ref{proposition_completion_Zn} that will be convenient further. Let us set
$$J=\{p\mid p \text{ does not divide } n\}, \hspace{1cm} S=\mathbb Z[J^{-1}].$$
Then we have the isomorphism $S/n\cong \mathbb Z/n.$ Consider the ideal
$$\mathcal{I}_S={\rm Ker}(S[A]\epi \mathbb Z/n).$$
Then we have the isomorphism $S[A]\cong (\mathbb
Z[A])[J^{-1}]\cong S\otimes \mathbb Z[A]$ and the ideal
$\mathcal{I}_S$ corresponds to $I_n[J^{-1}].$ Since elements of
$J$ are invertible modulo $n^m,$ we get the isomorphism
$S[A]/\mathcal{I}_S^m\cong \mathbb Z[A]/I_n^m\cong S\otimes
\mathbb Z[A]/I_n^m.$ It follows that $\widehat{S[A]}\cong
\widehat{\mathbb Z[A]}_{I_n},$ where
$\widehat{S[A]}=\widehat{S[A]}_{\mathcal{I}_S}.$ Similarly, we
have $\hat M_{I_n}\cong (M\otimes S)^{\wedge}_{\mathcal{I}_S}.$
Using the isomorphism $\widehat{S[A]}\cong \widehat{\mathbb
Z[A]}_{I_n},$ we can write the homomorphism $\theta_R$ as follows
\begin{equation}\label{eq_theta_2}
\theta_R :\hat A_n \longrightarrow \widehat{S[A]}.
\end{equation}
Then we get the new version of Proposition \ref{proposition_completion_Zn}:

\begin{Proposition}\label{proposition_completion_Zn}
Let $R$ be the ring $\mathbb Z/n$, $G$ be a finitely generated metabelian group as in \eqref{eq_MGA}. Then there is a short exact sequence
$$0 \longrightarrow (M\otimes S)^{\wedge}_{\mathcal{I}_S}\longrightarrow \hat G_R \longrightarrow \hat A_n  \longrightarrow 1,$$
whose morphisms are induced by \eqref{eq_MGA}, and the action by
conjugation of $\hat A_n$ on $(M\otimes
S)^{\wedge}_{\mathcal{I}_S}$ coincides with the action that
induced by the structure of $\widehat{S[A]}$-module via the
homomorphism $\theta_R:\hat A_n\to \widehat{S[A]}.$
\end{Proposition}

\begin{Proposition}\label{proposition_p-completion-q}
Let $G$ be a finitely generated  metabelian group, $p$ and $q$ be different prime numbers and $\hat G_p$ be the $p$-profinite completion of $G$. Then
$$H_2(\hat G_p,\mathbb Z/q)=0 \hspace{0.5cm} \text{ and } \hspace{0.5cm} H_2(G/\gamma^{[p^s]}_m(G),\mathbb Z/q)=0,$$
for any $m,s\geq 1.$
\end{Proposition}
\begin{proof} We prove the first equality. The second equality can be proved similarly.
By \eqref{eq_p-profinite} and Proposition
\ref{proposition_completion_Zn}, we get the metabelian
decomposition $\hat M_{I_p} \mono \hat G_p \epi \hat A_p.$
Consider the corresponding Lyndon-Hochschild-Serre spectral
sequence $E.$ It is sufficient to prove that $E^2_{i,j}=0$ for
$(i,j)\in \{(0,2),(1,1),(2,0)\}.$ Since $E^2_{0,2}=H_0(\hat A_p,
H_2( \hat M_{I_p},\mathbb Z/q)),$ using the universal coefficient
theorem and the equality $H_2(\hat M_{I_p})=\wedge^2 \hat
M_{I_p}$, we obtain the following exact sequence $$  (\wedge^2
\hat M_{I_p}\otimes \mathbb Z/q)_{\hat A_p} \to E^2_{0,2} \to {\sf
Tor}(\hat M_{I_p},\mathbb Z/q)_{\hat A_p} \to 0.$$  The groups
$M/(M*I_p^m)$ are quotients of $M/Mp^m$, and hence they are
uniquely $q$-divisible. It follows that $\hat M_{I_p}=\ilimit
M/(M*I_p^m)$ is an uniquely $q$-divisible abelian group. Thus
$\hat M_{I_p}\otimes \mathbb Z/q=0,$ $  (\wedge^2 \hat
M_{I_p}\otimes \mathbb Z/q)_{\hat A_p}=0$ and ${\sf Tor}(\hat
M_{I_p},\mathbb Z/q)=0.$ Therefore, $E^2_{0,2}=0$  and
$$E^2_{1,1}=H_1(\hat A_p, H_1(\hat M_{I_p},\mathbb Z/q))=H_1(\hat
A_p, \hat M_{I_p}\otimes \mathbb Z/q)=0.$$ Similarly, using the
universal coefficient theorem, we get $$E^2_{2,0}=H_2(\hat A_p,
H_0(\hat M_{I_p},\mathbb Z/q))=H_2(\hat A_p,\mathbb Z/q)=0.$$
\end{proof}

\section{Homology of an abelian group with coefficients.}

In this section $K$ denotes a commutative Notherian ring, $A$
denotes a finitely generated abelian group and $M$ denotes a
finitely generated $K[A]$-module. Since $A$ is a finitely
generated abelian group, $K[A]$ is a commutative Notherian ring.
Denote by $I=I_K(A)$ the augmentation ideal of $K[A].$ Then by
\eqref{eq_artin-rees-1} we have
\begin{equation}\label{eq_rees_cor}
MI^{\infty}\cdot I=MI^{\infty }
\end{equation}
We put $M_{\sf rn}=M/MI^{\infty},$ $\hat M=\hat M_I=\ilimit M/MI^i$ and $M^\ell=M^\ell_I=M[(1+I)^{-1}].$
\begin{Proposition}\label{prop_homology_completion_rn} Let $X$ be an abelian group. Then the homomorphisms $M\to M_{\sf rn}\to M^\ell \to \hat M$
induce isomorphisms
$$H_*(A,M\otimes X)\cong H_*(A,M_{\sf rn}\otimes X)\cong H_*(A,M^\ell\otimes X)\cong H_*(A,\hat M\otimes X)$$
$$H_*(A,{\sf Tor}(M,X))\cong H_*(A,{\sf Tor}(M_{\sf rn},X))\cong H_*(A,{\sf Tor}(M^\ell,X))\cong H_*(A,{\sf Tor}(\hat M,X))$$
and there is the following short exact sequence:
$$0 \longrightarrow {\varprojlim}^1 H_{m+1}(A,M/MI^i) \longrightarrow
H_m(A,M)\longrightarrow  \ilimit H_m(A,M/MI^i)\longrightarrow 0,$$ where the epimorphism is induced  by the projections $M\epi M/MI^i$.
\end{Proposition}
\begin{proof}
It follows from Lemma \ref{lemma_tor} and Proposition \ref{proposition_tor_lim_gen}.
\end{proof}
\begin{Corollary}
If $M=MI,$ then $H_*(A,M)=0.$
\end{Corollary}

\begin{Corollary}\label{cor_iso_lim_H}
If $K$ is an Artinian commutative ring, then the projections $M\epi M/MI^i$ induce the isomorphism
$$H_*(A,M)\cong \ilimit H_*(A,M/MI^i).$$
\end{Corollary}
\begin{proof}
The homology groups $H_{m+1}(A,M/MI^i)$ are finitely generated
$K$-modules, and hence they are Artinian $K$-modules. It follows
that the Mittag-Leffler condition holds for the inverse sequence
$H_{m+1}(A,M/MI^i)$, and hence ${\varprojlim}^1
H_{m+1}(A,M/MI^i)=0.$
\end{proof}

\begin{Lemma}\label{lemma_homology_isom_nilpotent_quotient}
Let $p$ be a prime number, ${\sf char}(K)=p,$ $M$ be a nilpotent
$K[A/A^{p^i}]$-module such that $MI^{p^m-1}=0$. Then the
projection $A\epi A/A^{p^i}$ induces an isomorphism
$$H_*(A,M)\cong H_*(A/A^{p^i},M).$$
\end{Lemma}
\begin{proof} First we prove it for the first homology.
Consider the short exact sequence $I(A)\mono K[A] \epi K$. The
associated long exact sequence gives us the four term exact
sequence $$H_1(A,M)\mono  M\otimes_{K[A]} I(A)\to M \epi M_A.$$
Similarly, we get the same sequence for $A/A^{p^i}.$ Since,
$I(A/A^{p^i})\otimes_{K[A/A^{p^i}]} M=I(A/A^{p^i})\otimes_{K[A]}
M$ and $M_A\cong M_{A/A^{p^i}},$ we obtain the morphism of exact
sequences:
$$
\begin{tikzpicture}
  \matrix (m)
 [matrix of math nodes,row sep=1cm,column sep=1.5cm,minimum width=2em]
{
  H_1(A,M) &M\otimes_{K[A]} I(A)  & M & M_A  \\
  H_1(A/A^{p^m},M) & M\otimes_{K[A]}  I(A/A^{p^i}) & M & M_{A}.      \\
   };
 \path[>->]
(m-1-1.east|-m-1-2) edge  (m-1-2)
(m-2-1.east|-m-2-2) edge  (m-2-2)
;
\path[->]
(m-1-2.east|-m-1-3) edge (m-1-3)
(m-2-2.east|-m-2-3) edge (m-2-3)
(m-1-1) edge (m-2-1)
(m-1-2) edge (m-2-2)
;
\path[->>]
(m-1-3.east|-m-1-4) edge (m-1-4)
(m-2-3.east|-m-2-4) edge (m-2-4)
;
\path[=]
(m-1-3) edge [double distance=3pt] (m-2-3)
(m-1-4) edge [double distance=3pt] (m-2-4)
;
\end{tikzpicture}$$
It is sufficient to prove that the morphism $M
\otimes_{K[A]}I(A)\to M\otimes_{K[A]} I(A/A^{p^i})$ is an
isomorphism. Since, the functor $M\otimes_{K[A]}-$ is right exact,
it is an epimorphism. The kernel of this morphism is generated by
elements of the form $m\otimes   (a^{p^i}-1).$ They are equal to
zero, because $MI^{p^i-1}=0,$ $a^{p^i}-1=(a-1)^{p^i}$
 and $m\otimes (a-1)^{p^i}=m(a-1)^{p^i-1}\otimes (a-1)=0.$

Now we generalize it for $H_k$ using induction by $k.$ Assume that
Lemma holds for $k-1.$ We have $M\otimes_{K[A]}I(A)\cong
M\otimes_{K[A]}I(A/A^{p^i}).$ Then using shift in homology and the
assumption we get $H_k(A,M)\cong
H_{k-1}(A,M\otimes_{K[A]}I(A))\cong
H_{k-1}(A,M\otimes_{K[A]}I(A/A^{p^i}))\cong
H_{k-1}(A/A^{p^i},M\otimes_{K[A]}I(A/A^{p^i}))\cong
H_{k}(A/A^{p^i},M).$
\end{proof}

\section{Notation and unification.}\label{section_notation}

In this section we introduce the notation that we use in the rest
of the paper. $A$ denotes a finitely generated abelian group, $M$
denotes a finitely generated $\mathbb Z[A]$-module, $R$ denotes a
fixed ring of the form $\mathbb Z[J^{-1}]$ or $\mathbb Z/n$,
$R[A]$ denotes the group algebra of $A$ over $R,$ and $I_R$
denotes the augmentation ideal of $R[A].$ Moreover, we denote
$$J=
\left\{
\begin{array}{l l}
J, &  \text{ if }  R=\mathbb Z[J^{-1}]\\
\{p\ \mid \ p \nmid n\}, & \text{ if }  R=\mathbb Z/n
\end{array}\right.,
$$
$$
S=\mathbb Z[J^{-1}],
\hspace{1cm}
\mathcal{I}_S=
\left\{
\begin{array}{l l}
I_R, & \text{ if }  R=\mathbb Z[J^{-1}]\\
I_S, & \text{ if }  R=\mathbb Z/n
\end{array}\right. ,
$$
$$\hat M_R=(M\otimes R)^{\wedge}_{I_R},
\hspace{1cm} \hat M_S=(M\otimes S)^{\wedge}_{\mathcal{I}_S}.$$
$$M_R^{\ell}=(M\otimes R)[(1+I_R)^{-1}],\hspace{1cm}  M_S^{\ell}=(M\otimes S)[(1+\mathcal{I}_S)^{-1}].$$

Observe that there is the unique epimorphism $S\epi R.$ In general
$\mathcal{I}_S$ is not the augmentation ideal of $S[A],$ but the
epimorphism $S[A]\epi R[A]$ takes $\: \mathcal{I}_S$ to $I_R.$ It
follows that there is a continuous epimorphism of completions
$\widehat{S[A]}\epi \widehat{R[A]},$ and the epimorphism of
localizations $S[A]^{\ell}\epi R[A]^{\ell},$ where
\begin{align*}
& \widehat{S[A]}:=\widehat{S[A]}_{\mathcal{I}_S},\\
& \widehat{R[A]}:=\widehat{R[A]}_{I_R},\\
& S[A]^\ell:=S[A][(1+\mathcal{I}_S)^{-1}],\\
& R[A]^\ell:=R[A][(1+I_R)^{-1}].
\end{align*}
By $\hat A_R$ we denote $A\otimes R$ in the case of $R=\mathbb
Z[J^{-1}],$ and $\hat A_n=\ilimit A/A^{n^i}$ in the case of
$R=\mathbb Z/n.$

We denote by $\sigma: S[A] \to S[A]$ the standard antipode i.e.
the $S$-linear map with $\sigma(a)=a^{-1}$ for $a\in A$. In
\eqref{eq_theta_1} and \eqref{eq_theta_2} we defined the
multiplicative homomorphism
$$\theta_R: \hat A_R\longrightarrow \widehat{S[A]},$$
such that $\hat\sigma(\theta_R(x))=\theta_R(x^{-1}).$
Then we can consider $\hat M_S$  a $\mathbb Z \hat  A_R$-module.
By $G$ we denote a finitely generated metabelian group with a metabelian decomposition
\begin{equation}\label{eq_met_decomp}
0 \longrightarrow M\longrightarrow G \longrightarrow A \longrightarrow 1
\end{equation}
Further, we put
$$\gamma^R_i(G)=\left\{\begin{array}{ll}
\gamma_i(G),& \text{ if } R=\mathbb Z[J^{-1}]\\
\gamma^{[n]}_i(G),& \text{ if } R=\mathbb Z/n
\end{array}\right. ,
\hspace{1cm}
t^R_i(G)=\left\{\begin{array}{ll}
(G/\gamma_i(G))\otimes R,& \text{ if } R=\mathbb Z[J^{-1}]\\
G/\gamma^{[n]}_i(G),& \text{ if } R=\mathbb Z/n
\end{array}\right. $$
and
\begin{equation}\label{eq_def_M_cal_tA}
t^R_i(A)=\left\{\begin{array}{ll}
A\otimes R,& \text{ if } R=\mathbb Z[J^{-1}]\\
A/A^{n^i},& \text{ if } R=\mathbb Z/n
\end{array}\right.,\hspace{1cm}
\mathcal{M}^i=\left\{\begin{array}{ll}
(M\otimes R)/(\gamma_i(G)\otimes R),& \text{ if } R=\mathbb Z[J^{-1}]\\
M/(\gamma_i^{[n]}(G)\cap M),& \text{ if } R=\mathbb Z/n.
\end{array}\right.
\end{equation}
Then the $R$-completion of $G$ is defined as follows
$$\hat G_R= \ilimit \: t^R_i(G),$$
and there are short exact sequences
\begin{equation}\label{eq_short_exact_i}
0 \longrightarrow \mathcal M^i \longrightarrow t_i^R(G) \longrightarrow t_i^R(A) \longrightarrow 1.
\end{equation}
Then Propositions \ref{proposition_completion_ZJ} and \ref{proposition_completion_Zn} can be rehash as follows.

\begin{Proposition}\label{prop_comp_rehash}
There is a short exact sequence
$$0 \longrightarrow \hat M_S\longrightarrow \hat G_R \longrightarrow \hat A_R \longrightarrow 1,$$
whose morphisms are induced by the sequence \eqref{eq_met_decomp}
and the action by conjugation of $\hat A_R$ on $\hat M_S$
coincides with the action that induced by the structure of
$\widehat{S[A]}$-module via the homomorphism $\theta_R:\hat A_R\to
\widehat{S[A]}.$
\end{Proposition}
 By $K$ we denote an Artinian quotient ring of $R$. In other words,
\begin{enumerate}
\item if $R=\mathbb Q$ then $K=\mathbb Q;$
\item if $R=\mathbb Z[J^{-1}],$ then $K=\mathbb Z/n,$ where all the primes of $n$ do not lie in $J;$
\item if $R=\mathbb Z/n$ then $K=\mathbb Z/n',$ where $n'$ is a divisor of $n.$
\end{enumerate}
Then we have epimorphisms
$$S\epi R\epi K.$$ Denote by $I_K$ the augmentation ideal of the group algebra $K[A]$.
Then the epimorphism $R[A]\epi K[A]$ takes $I_R$ to $I_K,$ and
hence we have the continuous epimorphisms of completions
$\widehat{S[A]}\epi \widehat{R[A]} \epi \widehat{K[A]}$ and the
epimorphisms of localizations ${S[A]}^{\ell}\epi {R[A]}^{\ell}
\epi {K[A]}^{\ell}.$ Further, we denote
$$M_K=M\otimes K, \hspace{1cm} \hat M_K=(M_K)^{\wedge}_{I_K}, \hspace{1cm} M_K^{\ell}=M_K[(1+I_K)^{-1}].$$
Note that, since $M$ is finitely generated, we have the isomorphisms
\begin{equation}\label{eq_hat_M_K_iso}
\hat M_K\cong \hat M_S\otimes K, \hspace{1cm}  M_K^{\ell}\cong  M_S^{\ell}\otimes K.
\end{equation}
By $N$ we denote an arbitrary finitely generated $K[A]$-module.
Endow the module $\hat N=\hat N_{I_K}$ by the structure of
$\mathbb Z[\hat A_R]$-module using the homomorphism $\theta_R:\hat
A_R\to \widehat{S[A]}$ and $\widehat{S[A]}\epi \widehat{K[A]}$.
Then by \eqref{eq_coinvariants} we get
\begin{equation}\label{eq_coinvar_2}
N_A\cong \hat N_A=\hat N_{\hat A_R}\cong \hat
N\otimes_{\widehat{K[A]}} K,
\end{equation}
where $(-)_A=H_0(A,-)$ and $(-)_{\hat A_R}=H_0(\hat A_R,-).$

For $K$-modules $N_1,N_2$ there is an isomorphism $N_1\otimes_K
N_2\cong N_1\otimes N_2,$ where $\otimes=\otimes_{\mathbb Z}.$ The
same holds for $S$ and $R.$ It follows that  $\wedge^2_{\sigma_K}
N \cong \wedge^2_{\sigma} N$ for an $K[A]$-module $N.$ From the
other side, it is easy to check that $\wedge^2_\sigma M \cong
(\wedge^2 M)_A.$  Then we have
\begin{equation}\label{eq_wedge_isom1}
\wedge^2_{\sigma_K} N \cong \wedge^2_{\sigma} N\cong (\wedge^2 N)_A, \hspace{1cm} \wedge^2_\sigma M \cong  (\wedge^2 M)_A.
\end{equation}
For abelian groups $M_1,M_2$ there is an isomorphism $(M_1\otimes K)\otimes_R(M_2\otimes K)\cong (M_1\otimes M_2)\otimes K.$
It implies the isomorphism $\wedge^2_{\sigma_K} M_K\cong (\wedge^2_\sigma M)\otimes K.$
Then we get the isomorphism
\begin{equation}\label{eq_wedge_isom2}
(\wedge^2 M_K)_A  \cong (\wedge^2 M)_A \otimes K.
\end{equation}

\section{Exterior squares and tame modules.}

Remind the outcome of \cite{Bieri-Strebel_1978},
\cite{Bieri-Strebel_1981} concerned with tame modules. A valuation
of the group $A$ is a homomorphism $v:A\to \mathbb R$ into the
additive group of $\mathbb R.$ The valuation monoid of $v$ is the
submonoid  $A_v=\{a\in A\mid v(a)\geq 0\}.$ The group of
valuations ${\sf Hom}(A,\mathbb R)$ has the natural structure of a
real vector space and quotient space  ${\sf S}(A)=\big({\sf
Hom}(A,\mathbb R)\setminus \{0\}\big)/\mathbb R_+$ is called the
valuation sphere of $A$.

Let $M$ be a finitely generated $\mathbb Z[A]$-module. The
Bieri--Strebel invariant of $M$ is the set $\Sigma(M)\subseteq
{\sf S}(A)$ consisted of rays $[v]$ such that $M$ is a finitely
generated $A_v$-module. The equality $\Sigma(M)={\sf S}(A)$ holds
if and only if $M$ is a finitely generated as an abelian group
\cite[theorem 2.1]{Bieri-Strebel_1978}. The module $M$ is said to
be tame if $\Sigma(M)\cup (-\Sigma(M))={\sf S}(A).$ The main
result of the article \cite{Bieri-Strebel_1978} says that  $G$ is
finitely presented if and only if  $M$ is a tame $A$-module.
Moreover, it is proved in \cite{Bieri-Strebel_1978} that
\begin{equation}\label{eq_ann}
\Sigma(M)=\Sigma(\mathbb Z[A]/{\sf Ann}\: M)
\end{equation}
and that there is an implication
\begin{equation}\label{eq_incl}
{\sf Ann}\: M_1\subseteq {\sf Ann}\: M_2\ \ \Rightarrow \ \  \Sigma(M_1)\subseteq \Sigma(M_2),
\end{equation}
where ${\sf Ann}={\sf Ann}_{\mathbb Z[A]}.$ For finitely generated
$A$-modules $M_1$ and $M_2$ the inclusions ${\sf
Ann}(M_1\otimes_{\mathbb Z[A]} M_2)\supseteq {\sf Ann}\: M_1$ and
${\sf Ann}(M_1\otimes_{\mathbb Z[A]} M_2)\supseteq {\sf Ann}\:
M_2$ and the implication  \eqref{eq_incl}  imply the inclusion
$\Sigma(M_1\otimes_{\mathbb Z[A]} M_2) \supseteq \Sigma(M_1) \cup
\Sigma(M_2)$ and in particular
\begin{equation}\label{eq_union_ideal}
\Sigma(\mathbb Z[A]/(\mathfrak{a}+\mathfrak{b})) \supseteq
\Sigma(\mathbb Z[A]/\mathfrak{a}) \cup \Sigma(\mathbb
Z[A]/\mathfrak{b})
\end{equation}
for any ideals $\mathfrak{a},\mathfrak{b}\triangleleft \mathbb
Z[A].$ The $\mathbb Z[A]$-module $M$ is finitely generated over
$A_v$ if and only if $M_\sigma$ is finitely generated over
$A_{-v}.$ It follows that
\begin{equation}\label{eq_sigma_Sigma}
 \Sigma(M_\sigma)=-\Sigma(M).
\end{equation}

\begin{Lemma}\label{lemma_finitely_gen_H2} If $M$ is a tame $\mathbb Z[A]$-module, then
\begin{itemize}
\item  $(\wedge^2 M)_A$ is a finitely generated abelian group,
\item $H_2(M,K)_A$ is a finitely generated $K$-module.
\end{itemize}
\end{Lemma}
\begin{proof}
Since $\Sigma((\wedge^2 M)_A)=\Sigma(\wedge^2_\sigma M)\supseteq \Sigma(M)\cup (-\Sigma(M))={\sf S}(A),$
we get that $(\wedge^2 M)_A$ is a finitely generated abelian group. Then $(\wedge^2 M_K)_A$ is a finitely generated $K$-module.
The exact sequence $(\wedge^2 M_K)_A \to H_2(M,K)_A \to (M_K)_A \to 0$ implies that $H_2(M,K)_A$ is an extension of finitely generated
$K$-modules, and hence it is finitely generated itself.
\end{proof}

\begin{Proposition}\label{prop_wedge_compl_loc}
Let $R=\mathbb Z[J^{-1}]$ or $R=\mathbb Z/n,$ $M$ be a tame
$\mathbb Z[A]$-module, $K$ be an Artinian quotient ring of $R$.
Then for $m>\!>0$ there are isomorphisms
$$(\wedge^2 \hat M_K)_{\hat A_R} = (\wedge^2 \hat M_K)_A \cong (\wedge^2   M_K^{\ell})_A \cong (\wedge^2  (M_K/M_KI_K^m))_A\cong
(\wedge^2 \mathcal M_K^m)_{t_m^R(A)},$$
where $\mathcal M_K^m=\mathcal M^m\otimes K.$
\end{Proposition}
\begin{proof}
First we prove that the assumptions of the Proposition
\ref{prop_tensor} and Corollary \ref{cor_wedge2} for the ring
$K[A]$ the ideal $I_K$ and the module $M_K$ hold. We need to prove
that $${\sf cl}({\sf Ann}_{K[A]}\: M_K + \sigma_{K}({\sf
Ann}_{K[A]}\: M_K))\supseteq I_{K}^{m_0}$$ for some $m_0\in
\mathbb N.$ Denote \begin{align*} & \aa_K:={\sf Ann}_{K[A]}\: N +
\sigma_{K}({\sf Ann}_{K[A]}\: N)\\ & \aa:={\sf Ann}\: M +
\sigma({\sf Ann}\: M).\end{align*}
Using that $M$ is tame,
\eqref{eq_union_ideal} and \eqref{eq_sigma_Sigma} we get
$$\Sigma(\mathbb Z[A]/ \aa)\supseteq \Sigma(\mathbb Z[A]/ {\sf
Ann}\: M)\cup (-\Sigma(\mathbb Z[A]/ {\sf Ann}\: M))={\sf S}(A).$$
Thus $\mathbb Z[A]/ \aa$ is a finitely generated abelian group.
Since the map $(\mathbb Z[A]/\aa)\otimes K\epi K[A]/\aa_K$ is an
epimorphism, $K[A]/ \aa_K$ is a finitely generated $K$-module.
Using that $K$ is an Artinian ring, we get that $K[A]/\aa_K$ is an
Artinian $K$-module, and hence the sequence $\aa_K+I_{K}^m$
stabilizes. Therefore, ${\sf cl}(\aa_K)\supseteq I_{K}^{m_0}$ for
some $m_0\in \mathbb N.$

Hence, by Corollary \ref{cor_wedge2}, for $m>\!>0,$ we have the
isomorphisms $$(\wedge^2 \hat M_K)_A \cong (\wedge^2 M_K^{\ell})_A
\cong (\wedge^2  (M_K/M_KI_K^m))_A.$$  Using Lemma
\ref{lemma_inclusions_n} and \eqref{eq_inclusion_of_towers_ZJ} we
get epimorphisms $M_K/M_KI_K^{s(m)}\epi \mathcal M^{t(m)}_K \epi
M_K/M_KI_K^{m}$ for some sequences $s(m),t(m)$ that converge to
infinity. For a big enough $m$ the epimorphism
$M_K/M_KI_K^{s(m)}\epi M_K/M_KI_K^{m}$ induces an isomorphism
$$(\wedge^2  (M_K/M_KI_K^{s(m)}))_A\cong (\wedge^2
(M_K/M_KI_K^m))_A.$$ It follows that the epimorphism $\mathcal
M^{t(m)}_K \epi M_K/M_KI_K^{m}$ induces an isomorphism $$(\wedge^2
\mathcal M^m_K)_A\cong (\wedge^2  (M_K/M_KI_K^m))_A\cong (\wedge^2
\hat M_K)_A$$ for $m>\!>0.$

Note that $(\wedge^2  \mathcal M^m_K)_{t_m^R(A)}=(\wedge^2
\mathcal M^m_K)_{\hat A_R}.$ Then we only need to prove that
$(\wedge^2 \hat M_K)_{\hat A_R}=(\wedge^2 \hat M_K)_A.$ For this
it is sufficient to prove that for $x\in \hat A_R$ and $m_1\wedge
m_2 \in (\wedge^2 \hat M_K)_A$ the identity $m_1x\wedge
m_2x=m_1\wedge m_2$ holds. The abelian group $(\wedge^2 \hat
M_K)_A$ is a quotient of $\hat M_K \otimes_{K[A]} (\hat
M_K)_{\sigma}.$ By Proposition \ref{prop_tensor} we have  $\hat
M_K \otimes_{K[A]} (\hat M_K)_{\sigma}=\hat M_K
\otimes_{\widehat{K[A]}} (\hat M_K)_{\hat \sigma}$ and by
Proposition \ref{prop_comp_rehash} the action of $\hat A_R$ on
$\hat M_K$ is induced by the structure of $\widehat{S[A]}$-module
via the homomorphism $\theta_R:\hat A_R\to \widehat{S[A]}$ and
$\widehat{S[A]}\epi \widehat{K[A]}.$ Hence, $$m_1x\wedge
m_2x=m_1\theta_R(x)\wedge m_2\theta_R(x)=m_1\wedge m_2\hat
\sigma(\theta_R(x))\theta_R(x)=$$
$$=m_1\wedge m_2\theta_R(x^{-1})\theta_R(x)=m_1\wedge m_2.$$
\end{proof}

\section{The limit formula.}

\begin{Theorem}\label{theorem_limit}
Let $G$ be a finitely presented metabelian group, $R=\mathbb
Z[J^{-1}]$ or $R=\mathbb Z/n$ and $K$ be an Artinian quotient ring
of $R$. Then the homomorphisms $\hat G_R\to t^R_i(G)$ and
$G/\gamma^R_i\to t^R_i(G)$ induce the isomorphisms
$$H_2(\hat G_R,K)\cong \ilimit H_2(t^R_i(G),K)\cong \ilimit H_2(G/\gamma^R_{i}(G),K).$$
\end{Theorem}
\begin{Corollary}\label{cor_limit}
Let $G$ be a finitely presented metabelian group, and $p$ is a
prime number. Then $$H_2(\hat G_p,\mathbb Z/p)\cong H_2^{\sf
cont}(\hat G_p,\mathbb Z/p).$$
\end{Corollary}
\begin{proof}
If follows from Theorem \ref{theorem_limit} and the formula
$H^{\sf cont}_*(\hat G_p,\mathbb Z/p)\cong \ilimit
H_*(G/\gamma^{[p]}_i(G),\mathbb Z/p)$ \cite{Serre}.
\end{proof}
\begin{Corollary}
Let $R=\mathbb Z[J^{-1}]\ne\mathbb Q $ and $n\in \mathbb N,$ such
that the prime divisors of $n$ do not lie in $J.$ Then $$H_2(\hat
G_R,\mathbb Z/n)\cong H_2(\hat G_{\mathbb Z},\mathbb Z/n).$$
\end{Corollary}
\begin{proof}[Proof of Theorem \ref{theorem_limit}]
First we note that on the category of finitely generated
$K$-modules the functor $$\ilimit: {\sf mod}(K)^{\omega^{\rm
op}}\to {\sf mod}(K)$$ is an exact functor, because the
Mittag-Leffler condition holds for all inverse sequences. In the
proof we use $\ilimit$  only in this category, and we use the
exactness. Further, by Proposition \ref{prop_finite_module} we
have $H_2(t^R_i(G),K)\cong H_2(G/\gamma^R_i(G),K).$ Hence we only
need to prove the isomorphism $H_2(\hat G_R,K)\cong \ilimit
H_2(t^R_i(G),K).$

We reduce the theorem to a prime $n$ in the case of $R=\mathbb
Z/n.$ Assume that the theorem holds for the case of $R=\mathbb
Z/p$, where $p$ is prime. Let now $R=\mathbb Z/n$, where
$n=p_1^{s_1}\cdot{}\dots{}\cdot p_l^{s_l}$. Then we have
isomorphisms $$\hat G_{\mathbb Z/n}\cong \prod \hat G_{\mathbb
Z/p_j} \ \ \text{and} \ \  G/{\gamma^{[n]}_i(G)}\cong \prod_j
G/\gamma^{\left[p_j^{s_j}\right]}_i(G)$$  \cite[12.3]{Bousfield}.
Let $p$ be one of the prime divisors.  Since $\mathbb Z/p$ is a
field, we have the isomorphism $$H_2(\hat G_{\mathbb Z/n},\mathbb
Z/p)\cong \bigoplus_{i_1+\dots+i_l=2} \bigotimes_{j=1}^l
H_{i_j}(\hat G_{\mathbb Z/p_j},\mathbb Z/p).$$ If $p_j\ne p,$ by
Proposition \ref{proposition_p-completion-q} we have $H_2(\hat
G_{\mathbb Z/p_i},\mathbb Z/p)=0$ and obviously $H_1(\hat
G_{\mathbb Z/p_i},\mathbb Z/p)=0.$  Thus $H_2(\hat G_{\mathbb
Z/n},\mathbb Z/p)\cong  H_2(\hat G_{\mathbb Z/p},\mathbb Z/p).$
Similarly we get $$H_2( G/\gamma^{[n]}_i(G),\mathbb Z/p)\cong
H_2(G/\gamma^{[p^s]}_i(G),\mathbb Z/p).$$ Since
$\gamma^{[p]}_i(G)\supseteq \gamma^{[p^s]}_i(G)\supseteq
\gamma^{[p]}_{i'}(G)$ for any $i,$ we get $\ilimit
H_2(G/\gamma^{[p^s]}_i(G),\mathbb
Z/p)=H_2(G/\gamma^{[p]}_i(G),\mathbb Z/p).$ Then we obtain the
isomorphism $$H_2(\hat G_{\mathbb Z/n},\mathbb Z/p)\cong \ilimit
H_2( G/\gamma^{[n]}_i(G),\mathbb Z/p).$$ Then the theorem holds
for $R=\mathbb Z/n$ and $K=\mathbb Z/p,$ where $p$ is a divisor of
$n.$ Further, using the short exact sequence $\mathbb Z/p^{i}\mono
\mathbb Z/p^{i+1} \epi \mathbb Z/p,$ and the associated long exact
sequence of homology $H_*(\hat G_{\mathbb Z/n},-)$ and
$H_*(G/\gamma^{[n]}_i(G),-)$ by induction we get the theorem for
$R=\mathbb Z/n$ and $K=\mathbb Z/p^m.$ Finally, for $K=\mathbb
Z/n'=\bigoplus \mathbb Z/p_i^{m_i},$ we have $H_2(\hat G_{\mathbb
Z/n},K)=\bigoplus H_2(\hat G_{\mathbb Z/n},\mathbb Z/p_i^{m_i}).$
Therefore, the theorem holds in the general case for $R=\mathbb
Z/n.$ Further, we will assume that in the case of $R=\mathbb Z/n$
that $n=p$ is prime.

Let $E$ be a first quadrant homological spectral sequence that
converges to $\mathcal H_*$. If we are interested only in
$\mathcal H_m$ for $0\leq m\leq 2$ it is convenient to cut off the
spectral sequence as follows:
$$\tilde E^{r}_{pq}:=\left\{
\begin{array}{ll}
E^r_{pq}, & \text{ if } q\in \{0,1\} \text{ or } (p,q)=(0,2) \\
0, & \text{ else}
\end{array}\right.$$ with the obvious differentials.  Then $\tilde E$ has a limit $\tilde{\mathcal H}_*$ such that
$$\tilde{\mathcal H}_m\cong  \mathcal H_m, \text{ \ for } 0\leq m\leq 2.$$

Consider the morphism of metabelian decompositions:
\begin{equation} \label{eq_morphism_of_met}
\begin{tikzpicture}
  \matrix (m)
 [matrix of math nodes,row sep=0.5cm,column sep=1cm,minimum width=2em]
  { 0 &  {\hat M}_S & \hat G_R & \hat A_R & 1 \\
  0 & \mathcal M^i & t^R_i(G)  &  t^R_i(A) & 1. \\};
  \path[-stealth]
(m-1-1.east|-m-1-2) edge (m-1-2)
(m-1-2.east|-m-1-3) edge  (m-1-3)
(m-1-3.east|-m-1-4) edge  (m-1-4)
(m-1-4.east|-m-1-5) edge  (m-1-5)
(m-2-1.east|-m-2-2) edge  (m-2-2)
(m-2-2.east|-m-2-3) edge  (m-2-3)
(m-2-3.east|-m-2-4) edge  (m-2-4)
(m-2-4.east|-m-2-5) edge  (m-2-5)
(m-1-2) edge (m-2-2)
(m-1-3) edge (m-2-3)
(m-1-4) edge (m-2-4);
\end{tikzpicture}
\end{equation}
It induces the morphism of Lyndon-Hochschild-Serre spectral
sequences $E(\hat G_R)\to E(t^R_i(G))$ for homology with
coefficients in $K.$ If we cut off them we get the morphism
$\tilde E(\hat G_R)\to \tilde E(t^R_i(G)).$ We assume that these
spectral sequences start from the second page.

We prove that all the terms of the spectral sequence $E(t^R_i(G))$
are finitely generated $K$-modules. We only need to prove it for
the second page. For first two rows it is obvious and for
$(p,q)=(0,2)$ it is the statement of Lemma
\ref{lemma_finitely_gen_H2}.

Therefore  $\ilimit$ is exact on the terms of $\tilde
E(t^R_i(G)),$ and hence we can apply it to all the terms of
$\tilde E(t^R_i(G))$ and get a new spectral sequence
$$\tilde E^{\rm lim}=\ilimit \tilde  E(t^R_i(G)).$$
The morphisms $\tilde E(\hat G_R)\to \tilde E(t^R_i(G))$ induce the morphism
$$\tilde E(\hat G_R)\longrightarrow \tilde E^{\rm lim}.$$
In order to finish the prove it is sufficient to prove that the
morphism $\tilde E(\hat G_R)\to \tilde E^{\rm lim}$ is an
isomorphism of spectral sequences. It is enough to prove it on the
second pages. In other words we need to prove that the morphisms
\begin{equation}\label{eq_isom1}
H_*(\hat A_R, K)\longrightarrow \ilimit H_*( t^R_i(A), K),
\end{equation}
\begin{equation}\label{eq_isom2}
H_*(\hat A_R, M_K)\longrightarrow \ilimit H_*( t^R_i(A),\mathcal M^i_K),
\end{equation}
\begin{equation}\label{eq_isom3}
H_2(\hat M_S,K)_{\hat A_R}\longrightarrow \ilimit H_2(\mathcal M^i,K)_{t^R_i(A)}
\end{equation} are isomorphisms. Note that the homomorphism \eqref{eq_isom1} is a special case of \eqref{eq_isom2}.
Then need to prove that the homomorphism \eqref{eq_isom2} and the homomorphism \eqref{eq_isom3} are isomorphisms.
We prove it in Propositions \ref{proposition_isom2} and \ref{proposition_isom3}.
\end{proof}

We denote the natural homomorphisms
$$\tau_A:A\longrightarrow \hat A_R, \hspace{1cm} \tau_M: M \longrightarrow \hat M_S$$ and
$$A_1:={\rm Ker}(\tau_A), \hspace{1cm} A_2:={\rm Coker}(\tau_A).$$
By ${\sf t}_p(A)$ we denote the $p$-power torsion subgroup of $A.$

\begin{Lemma}\label{lemma_Ker_Coker_J}
\begin{enumerate}
\item $A_1=\bigoplus_{p\in J} {\sf t}_p(A).$ \item If $R=\mathbb
Z[J^{-1}]$, then  $A_2=(\mathbb Z[J^{-1}]/\mathbb Z)\otimes A$
\item If $R=\mathbb Z/p$, then  $A_2=(\mathbb Z_p/\mathbb
Z)\otimes A.$ Moreover, if $A$ is finite, then $A_2=0$, else
$A_2\cong \mathbb Q^{\oplus {\bf c}}$ .
\end{enumerate}
\end{Lemma}
\begin{proof}
The only non-obvious thing is the last isomorphism. First, we
prove that $\mathbb Z_p/\mathbb Z$ is a divisible abelian group.
Let $q\ne p$ be a prime. The group $\mathbb Z_p$ is $q$-divisible,
and hence $\mathbb Z_p/\mathbb Z$ is $q$-divisible. Then we need
to prove that $\mathbb Z_p/\mathbb Z$ is $p$-divisible. It follows
from the following  equality modulo $\mathbb Z$:
$\sum_{i=0}^{\infty} \alpha_ip^i= \sum_{i=1}^{\infty}
\alpha_ip^i=p(\sum_{i=1}^{\infty} \alpha_{i}p^{i-1}).$ Therefore,
$\mathbb Z_p/\mathbb{Z}$ is a divisible torsion-free group. Then
by description of divisible groups \cite[IV]{Fuchs} we get
$\mathbb Z_p/\mathbb{Z}\cong \mathbb Q^{\oplus {\bf c}}$. Then the
required statement follows immediately.
\end{proof}

\begin{Lemma}\label{lemma_kernel_cokernel_zero_homology} \
\begin{enumerate}
\item If $H$ a finite group such that $|H|^{-1}\in \mathbb
[J^{-1}]$ and $L$ is an $S[H]$-module, then $H_k(H,L)=0$ for
$k>0.$ \item If $L$ is a $S[A_1]$-module, then $H_k(A_1,L)=0$ for
$k>0.$ \item If $L$ is a residually nilpotent $K[A_2]$-module,
then $H_k(A_2,L)=0$ for $k>0.$
\end{enumerate}
\end{Lemma}
\begin{proof}
(1) The trivial $S[H]$-module $S$ is projective because it is
isomorphic to the direct summand  of $S[H]$ given by the image of
the projector $x\mapsto (\sum_{h\in H} xh)/|H|.$ Therefore, for
any $S[H]$-module $L$ we have  $H_k(H, L)={\sf Tor}_k^{S[H]}(S,
L)=0.$

(2) It follows from Lemma \ref{lemma_Ker_Coker_J} and (1).

(3) Let $R=\mathbb Z[J^{-1}].$ Then $A_2=(\mathbb
Z[J^{-1}]/\mathbb Z)\otimes A$. Since $\mathbb Z[J^{-1}]/\mathbb
Z=\bigoplus_{q\in J} \mathbb Z/q^{\infty},$ we get $A_2=(\mathbb
Z[J^{-1}]/\mathbb Z)^d\oplus B_0,$ where $B_0$ is a finite group
such that $|B_0|^{-1}\in \mathbb Z[J^{-1}].$ The group $\mathbb
Z[J^{-1}]/\mathbb Z$ is isomorphic to the direct limit
$\varinjlim\: \mathbb Z/j_i,$ where $j_i$ runs over natural
numbers with prime divisors in $J$ such that for any natural
number $j$ with prime divisors in $J$ we have $j\mid j_i$ for
$i>\!>0.$ Therefore, $A_2$ is isomorphic to the direct limit
$\varinjlim\: B_i,$ where $B_i$ is a finite abelian group such
that $|B_i|^{-1}\in \mathbb Z[J^{-1}].$ Using (1) and the
epimorphism $S\epi K$ we get $H_k(B_i,L)=0.$ Finally, using the
formula $H_k(A_2,L)=\varinjlim\: H_k( B_i,L)$ we get
$H_k(A_2,L)=0.$

Let $R=\mathbb Z/p.$ If $A_2=0$, the statement is obvious, then we
can assume $A_2\cong \mathbb Q^{\oplus {\bf c}}$ and $K=\mathbb
Z/p.$  Hence for an element $a$ of $A_2$ there is an element
$a_1\in A_2$  such that $a=a_1^p.$ Using the equality
$a-1=a_1^{p}-1= (a_1-1)^p\ {\sf mod }\ p$ we get
$LI_R=LI_R^p=LI_R^{\infty}=0,$ and hence the action of $A_2$ on
$L$ is trivial. Since $A_2$ is torsion-free, we have
$H_k(A_2)=\wedge^k A_2.$ Then $H_k(A_2)\otimes L=0$ and ${\sf
Tor}(H_{k-1}(A_2),L)=0,$ and hence by universal coefficient
theorem $H_k(A_2,L)=0.$
\end{proof}

\begin{Lemma}\label{lemma_hat_without}
The homomorphisms $\tau_A$ and $\tau_M$ induce isomorphisms
$$H_*(A,M_K)\cong H_*(A,\hat M_K)\cong  H_*(\hat A_R, \hat M_K).$$
\end{Lemma}
\begin{proof}
The first isomorphism we get by Proposition
\eqref{prop_homology_completion_rn}. Consider, the short exact
sequence $A_1\mono A \epi {\rm Im}(\tau_A),$ and the corresponding
Lyndon-Hochschild-Serre spectral sequence $H_i({\rm
Im}(\tau_A),H_j(A_1,\hat M_K)) \Rightarrow H_{i+j}(A,\hat M_K).$
By Lemma \ref{lemma_kernel_cokernel_zero_homology} we get
$H_j(A_1,\hat M_K)=0$ for $j>0.$ Moreover, since  $\hat M_K$ has
the natural structure of a $\hat A_R$-module that lifts the
structure of $A$-module, then $A_1$ acts trivially on $\hat M_K,$
and hence $H_0(A_1,\hat M_K)=\hat M_K.$ It follows that the
homomorphism $A\to {\rm Im}(\tau_A)$ induce the isomorphism
$H_*(A,\hat M_K)\cong H_*({\rm Im}(\tau_A),\hat M_K).$

Consider the short exact sequence ${\rm Im}(\tau_A)\mono \hat A_R
\epi A_2,$ and the corresponding Lyndon-Hochschild-Serre spectral
sequence $H_i(A_2,H_j({\rm Im}(\tau_A),\hat M_K)) \Rightarrow
H_{i+j}(\hat A_R,\hat M_K).$ Since $M_K$ is a finitely generated
$K[A]$-module, $M_K$ is a finitely generated $K[{\rm
Im}(\tau_A)]$-module. Using the formula
\eqref{prop_homology_completion_rn}  we get $H_j({\rm
Im}(\tau_A),\hat M_K)=H_j({\rm Im}(\tau_A), M_K).$  Using the
formula \eqref{eq_coinvar_2} and the fact the action of $A$ on
$H_j({\rm Im}(\tau_A),\hat M_K)$ is trivial, we get $$H_j({\rm
Im}(\tau_A),\hat M_K)=H_j({\rm Im}(\tau_A),\hat M_K)_A=H_j({\rm
Im}(\tau_A),\hat M_K)_{\hat A_R}.$$ Thus, $H_j({\rm
Im}(\tau_A),\hat M_K)$ is a trivial $\hat A_R$-module. Then by
Lemma \ref{lemma_kernel_cokernel_zero_homology} we get
$H_i(A_2,H_j({\rm Im}(\tau_A),\hat M_K))=0$ for $i>0$ and
$H_0(A_2,H_j({\rm Im}(\tau_A),\hat M_K))=H_j({\rm Im}(\tau_A),\hat
M_K).$ It follows that the homomorphism ${\rm Im}(\tau_A)\mono
\hat A_R$ induces the isomorphism $H_*({\rm Im}(\tau_A),\hat
M_K)\cong H_*(\hat A_R,\hat M_K).$ It follows that the morphism
$\tau_A$ induces the isomorphism $H_*(A,\hat M_K)\cong H_*(\hat
A_R,\hat M_K).$
\end{proof}

\begin{Lemma}\label{lemma_nilpotent_oba} If $N$ is a nilpotent $K[t^R_i(A)]$-module such that $NI^{i+1}=0$,
then the homomorphisms $A\to t^r_i(A)$ induce the isomorphism:
$$H_*(A,N)\cong H_*(t^R_i(A),N).$$
\end{Lemma}
\begin{proof}
If $R=\mathbb Z[J^{-1}],$ it follows from Lemma \ref{lemma_hat_without}. If $R=\mathbb Z/p$, it follows from \ref{lemma_homology_isom_nilpotent_quotient}.
\end{proof}

\begin{Proposition}\label{proposition_isom2}
The homomorphism \eqref{eq_isom2} is an isomorphism.
\end{Proposition}
\begin{proof}
Let $R=\mathbb Z[J^{-1}].$ Then $t^R_i(A)=\hat A_R=A\otimes R$ and
$\mathcal M^i_K=M_K/M_KI_K^i.$ Then by Lemma
\ref{lemma_hat_without} $H_*(\hat A_R, \hat M_K)\cong H_*(A,\hat
M_K),$ by Corollary \ref{corollary_shoert_eact _sequence_limit}
$H_*(A,\hat M_K)\cong \ilimit H_*(A,\mathcal M^i_K)$ and again by
Lemma \ref{lemma_hat_without}, using that $\mathcal M^i_K$ is
nilpotent, we get $H_*(A,\mathcal M^i_K)\cong
H_*(t^R_i(A),\mathcal M^i_K).$

Let $R=\mathbb Z/p.$ Then by Lemma
\ref{lemma_homology_isom_nilpotent_quotient}
$H_*(t^R_i(A),\mathcal M^i_K)\cong H_*(A,\mathcal M^i_K).$ Since,
the inverse sequence $\mathcal M^i_K$ is equivalent to the inverse
sequence $M/MI^i$ we have $\ilimit H_*(A,\mathcal M_i)\cong
\ilimit H_*(A,M/MI^i).$ By Proposition
\ref{prop_homology_completion_rn} we get $\ilimit
H_*(A,M/MI^i)\cong H_*(A,\hat M_K)$ and by Lemma
\ref{lemma_hat_without} $H_*(A,\hat M_K)\cong H_*(\hat A_R,\hat
M_K).$
\end{proof}

\begin{Proposition}\label{proposition_isom3}
The homomorphism \eqref{eq_isom3} is an isomorphism.
\end{Proposition}
\begin{proof}
Consider the morphism of exact sequences
\begin{equation}
\begin{tikzpicture}
  \matrix (m)
 [matrix of math nodes,row sep=0.8cm,column sep=0.7cm,minimum width=2em]
  { H_1(\hat A_R,{\sf Tor}(\hat M_S,K)) & (\wedge^2 \hat M_K)_{\hat A_R} & H_2(\hat M_S,K)_{\hat A_R} & (M_K)_A & 0 \\
  H_1(t^R_i(A),{\sf Tor}(\mathcal M^i,K)) & (\wedge^2 \mathcal  M_K^i)_{t^R_i(A)} & H_2(\mathcal M^i,K)_{t^R_i(A)} & (\mathcal M^i_K)_{t^R_i(A)} & 0 \\};
  \path[-stealth]
(m-1-1.east|-m-1-2) edge (m-1-2)
(m-1-2.east|-m-1-3) edge  (m-1-3)
(m-1-3.east|-m-1-4) edge  (m-1-4)
(m-1-4.east|-m-1-5) edge  (m-1-5)
(m-2-1.east|-m-2-2) edge  (m-2-2)
(m-2-2.east|-m-2-3) edge  (m-2-3)
(m-2-3.east|-m-2-4) edge  (m-2-4)
(m-2-4.east|-m-2-5) edge  (m-2-5)
(m-1-1) edge node [right] {$f_1^i$} (m-2-1)
(m-1-2) edge node [right] {$f_2^i$} (m-2-2)
(m-1-3) edge node [right] {$f^i$} (m-2-3)
(m-1-4) edge node [right] {$f_3^i$} (m-2-4);
\end{tikzpicture}
\end{equation} We need to prove that $\ilimit f^i$ is an isomorphism.
By Lemma \ref{lemma_tor} we have $H_*(\hat A_R,{\sf Tor}(\hat
M_S,K))\cong H_*(\hat A_R,{\sf Tor}( M_S,K)^{\wedge}).$ By Lemma
\ref{lemma_hat_without} we get $H_*(\hat A_R,{\sf Tor}(
M_S,K)^{\wedge})\cong H_*( A,{\sf Tor}( M_S,K)^{\wedge})\cong H_*(
A,{\sf Tor}( \hat M_S,K)).$ By Propositions
\ref{prop_tensor_tor_abelian_limit} and
\ref{proposition_tor_lim_gen} we obtain $H_*( A,{\sf Tor}( \hat
M_S,K))\cong H_*(A,\ilimit {\sf Tor}(\mathcal M^i,K))\cong \ilimit
H_*(A, {\sf Tor}(\mathcal M^i,K)),$ and by Lemma
\ref{lemma_nilpotent_oba} we obtain $ H_*(A, {\sf Tor}(\mathcal
M^i,K))\cong H_*(t^R_i(A),{\sf Tor}(\mathcal M^i,K)).$  Then
$\ilimit f_1^i$ and $\ilimit f_3^i$ are isomorphisms. By
Proposition \ref{prop_wedge_compl_loc} $\ilimit f_2^i$ is an
isomorphism. Finally, using the five lemma, we get that $\ilimit
f^i$ is an isomorphism.
\end{proof}

\section{Bousfield problem for metabelian groups.}

We put $\Phi^R_iH_2(G,K)={\rm Ker}(H_2(G,K)\to
H_2(G/\gamma^R_{i+1}(G),K)).$ Then $\Phi_iH_2(G,K)=\Phi^{\mathbb
Z}_iH_2(G,K)$ is the Dwyer filtration on $H_2(G,K)$ (see
\cite{Dw}).

\begin{Theorem}\label{Theorem_Bousfield}
Let $G$ be a finitely presented metabelian group, $R=\mathbb
Z[J^{-1}]$ or $R=\mathbb Z/n$ and $K$ be an Artinian quotient ring
of $R.$ Then for $i>\!>0$ there is a short exact sequence
$$0 \longrightarrow \Phi^R_iH_2(G,K)\longrightarrow H_2(G,K)\longrightarrow H_2(\hat G_R,K) \longrightarrow 0,$$ where
the epimorphism is induced by  the homomorphism $G\to \hat G_R.$
\end{Theorem}
The following corollary is the answer on the Bousfield problem for
the class of metabelian groups.

\begin{Corollary}\label{cor_bousfield}
Let $G$ be a finitely presented metabelian group. Then the homomorphisms $G\to \hat G_{\mathbb Q}$ and $G\to \hat G_{\mathbb Z/n}$ induce the epimorphisms
$$H_2(G,\mathbb Q)\epi H_2(\hat G_{\mathbb Q},\mathbb Q), \hspace{1cm} H_2(G,\mathbb Z/n)\epi H_2(\hat G_{\mathbb Z/n},\mathbb Z/n).$$
\end{Corollary}
\begin{Corollary}
Let $G$ be a finitely presented metabelian group and $\hat G$ be the pronilpotent completion. Then for $m>\!>0$ there is a short exact sequence
$$0 \longrightarrow \Phi_mH_2(G,\mathbb Z/n)\longrightarrow H_2(G,\mathbb Z/n)\longrightarrow H_2(\hat G,\mathbb Z/n) \longrightarrow 0,$$
where the epimorphism is induced by  the homomorphism $G\to \hat G.$
\end{Corollary}

\begin{proof}[Proof of Theorem \ref{Theorem_Bousfield}] For the sake of simplicity we put $\gamma^R_i=\gamma^R_i(G)$
Consider the short exact sequence
$$1\longrightarrow \gamma^R_i\longrightarrow  G \longrightarrow G/\gamma^R_i \longrightarrow 1,$$
 and the associated five term exact sequence
$$H_2(G,K)\to H_2(G/\gamma^R_i,K)\to H_1(\gamma^R_i,K)_{G}\to H_1(G,K)\epi H_1(G/\gamma^R_i,K).$$
Note that $H_1(G,K)\cong H_1(G/\gamma^R_i,K)\cong
(G/\gamma^R_2)\otimes K$ and the morphism $H_1(G,K)\to
H_1(G/\gamma^R_i,K)$ is the isomorphism. Moreover,
$$H_1(\gamma^R_i,K)_{G}=(\gamma^R_i/[\gamma^R_i, G])\otimes
K=(\gamma^R_i/\gamma^R_{i+1})\otimes K.$$ Hence, we get the exact
sequence:
$$0\longrightarrow \Phi^R_{i-1}H_2(G,K)\longrightarrow H_2(G,K) \overset{\xi_i}\longrightarrow H_2(G/\gamma^R_i,K)
\longrightarrow (\gamma^R_i/\gamma^R_{i+1})\otimes
K\longrightarrow 0.$$ The inclusion $\gamma^R_{i+1}\hookrightarrow
\gamma^R_i$ induce zero homomorphism
$(\gamma_{i+1}^R/\gamma^R_{i+2})\otimes K\to
(\gamma_i^R/\gamma_{i+1}^R)\otimes K.$ Therefore $\ilimit
(\gamma_i^R/\gamma^{R}_{i+1})\otimes K=0.$ It follows that
$\ilimit {\rm Im}(\xi_i)\cong \ilimit H_2(G/\gamma^R_i,K).$ By
Theorem \ref{theorem_limit} we get $\ilimit {\rm Im}(\xi_i)\cong
H_2(\hat G_R,K).$ Since $H_2(G,K)$ is an Artinian $K$-module, the
sequence $\Phi_j^RH_2(G,K)$ stabilizes and we get $\ilimit
\Phi_i^RH_2(G,K)=\bigcap_j \Phi_j^RH_2(G,K)=\Phi_i^RH_2(G,K)$ for
$i>\!>0.$ It follows that the image ${\rm Im}(\xi_i)$ stabilizes.
Hence ${\rm Im}(\xi_i)\cong H_2(\hat G_R,K)$ for $i>\!>0$ and we
have the short exact sequences.
$$0\longrightarrow H_2(\hat G_R,K)\longrightarrow H_2(G/\gamma^R_i,K) \longrightarrow (\gamma^R_i/\gamma^R_{i+1})\otimes K\longrightarrow 0$$
$$0 \longrightarrow \Phi^R_iH_2(G,K)\longrightarrow H_2(G,K)\longrightarrow H_2(\hat G_R,K) \longrightarrow 0,$$ for $i>\!>0.$
\end{proof}
\begin{Remark}
In the proof of Theorem \ref{Theorem_Bousfield} we get the short exact sequence
$$0\longrightarrow H_2(\hat G_R,K)\longrightarrow H_2(G/\gamma^R_i,K) \longrightarrow
(\gamma^R_i/\gamma^R_{i+1})\otimes K\longrightarrow 0$$ for
$i>\!>0.$ Then informally the group $H_2(\hat G_R,K)$ can be
considered as 'the biggest part' of $H_2(G/\gamma^R_i,K)$
independent of $i.$
\end{Remark}
Next we give an example of a polycyclic metabelian residually
nilpotent group $H$, such that the intersection of Dwyer
filtration $\cap_i \Phi_i^{\mathbb Z}(H)$ is nonzero (see
\cite{Mikhailov} for detailed study of this group and its
localizations).
\begin{Example} Let $H=\langle a,b\ |\ a^{b^2}=aa^{3b},\
[a,a^b]=1\rangle$. The group $H$ is the semidirect product
$(\mathbb Z\oplus \mathbb Z)\rtimes \mathbb Z$, where the cyclic
group $\mathbb Z=\langle b\rangle$ acts on $\mathbb Z\oplus
\mathbb Z$ as the matrix $\begin{pmatrix}
0 & 1 \\
1 & 3
\end{pmatrix}.
$
For the group $H$,
$$
\cap_i\Phi_i^{\mathbb Z}(H)=H_2(H)=\mathbb Z/2.
$$
\end{Example}

\section{The Telescope.}
In this section we assume that $R=\mathbb Z$ and $\hat G=\hat G_R$
is the pronilpotent completion of $G.$ Moreover, we assume that
$A=G_{ab}$ and $M=[G,G].$  In \cite{Levine1} and \cite{Levine2},
J.P. Levine defines closely related groups, his algebraic closure
of $G$, whose image in the pronilpotent completion has important
properties. In the case of metabelian  group $G$ this image is
called the Telescope of $G$ and denoted by $\bar G.$ It was proved
in \cite{Baumslag-Mikhailov-Orr} that the metabelian decomposition
$M\mono G \epi A$ induces the following metabelian decomposition
$$0 \longrightarrow M^\ell \longrightarrow \bar G \longrightarrow A \longrightarrow 1,$$
where $M^\ell=M[(1+I)^{-1}]$ is the localization of $M$ with
respect to the multiplicative set $1+I.$

\begin{Proposition}
Let $G$ be a finitely generated metabelian group and $n\geq 1$.
Then the inclusion $\bar G\to \hat G$ induces an isomorphism
$$H_2(\bar G, \mathbb Z/n)  \cong H_2(\hat G,\mathbb Z/n).$$
\end{Proposition}
\begin{proof}The morphism  of the metabelian decompositions
\begin{equation}
\begin{tikzpicture}
  \matrix (m)
 [matrix of math nodes,row sep=0.7cm,column sep=1cm,minimum width=2em]
  { 0 &  {M}^\ell & \bar G & A & 1 \\
  0 & \hat M & \ \hat G  & A & 1. \\};
  \path[-stealth]
(m-1-1.east|-m-1-2) edge (m-1-2)
(m-1-2.east|-m-1-3) edge  (m-1-3)
(m-1-3.east|-m-1-4) edge  (m-1-4)
(m-1-4.east|-m-1-5) edge  (m-1-5)
(m-2-1.east|-m-2-2) edge  (m-2-2)
(m-2-2.east|-m-2-3) edge  (m-2-3)
(m-2-3.east|-m-2-4) edge  (m-2-4)
(m-2-4.east|-m-2-5) edge  (m-2-5)
(m-1-2) edge  (m-2-2)
(m-1-3) edge  (m-2-3)
(m-1-4) edge node [right] {${\sf id}$} (m-2-4);
\end{tikzpicture}
\end{equation}
gives the morphism of the Lyndon-Hochschild-Serre spectral
sequences $E(\bar G)\to E(\hat G).$ First, we note that
$E^2_{i,0}(\bar G)=E^2_{i,0}(\hat G)=H_i(A,\mathbb Z/n).$ By
Proposition \ref{prop_homology_completion_rn} we have the
isomorphisms $H_i(A,M\otimes \mathbb Z/n)\cong H_i(A,M^\ell\otimes
\mathbb Z/n)\cong H_i(A,\hat M\otimes \mathbb Z/n).$ Therefore,
the morphism $E^2_{ij}(\bar G)\to E^2_{ij}(\hat G)$ is an
isomorphism for $j\in \{0,1\}.$  Then, it sufficient to prove that
the morphism $E^{2}_{0,2}(\bar G)=H_2(M^\ell,\mathbb  Z/n)_A\to
H_2(\hat M, \mathbb Z/n)_A =E^{2}_{0,2}(\hat G)$ is an
isomorphism.

By Proposition \ref{prop_homology_completion_rn}, we get
$H_i(A,{\sf Tor}(M^\ell,\mathbb Z/n))\cong H_i(A,{\sf Tor}(\hat
M,\mathbb Z/n))$ and by Proposition \ref{prop_wedge_compl_loc} we
get  $(\wedge^2 M^{\ell})_A\cong  (\wedge^2 \hat M)_A.$ Consider
the morphism of exact sequences
\begin{equation}
\begin{tikzpicture}
  \matrix (m)
 [matrix of math nodes,row sep=0.7cm,column sep=1cm,minimum width=2em]
  { H_1(A,{\sf Tor}(M^\ell,\mathbb Z/n)) &  (\wedge^2 M^\ell)_A & H_2(M^\ell,\mathbb Z/n) & {\sf Tor}(M^\ell,\mathbb Z/n)_A  \\
  H_1(A,{\sf Tor}(\hat M,\mathbb Z/n)) &  (\wedge^2 \hat M)_A & H_2(\hat M,\mathbb Z/n) & {\sf Tor}(\hat M,\mathbb Z/n)_A. \\};
  \path[->]
(m-1-1.east|-m-1-2) edge (m-1-2)
(m-1-2.east|-m-1-3) edge  (m-1-3)
(m-2-1.east|-m-2-2) edge  (m-2-2)
(m-2-2.east|-m-2-3) edge  (m-2-3)
(m-1-1) edge node [right] {$\cong$} (m-2-1)
(m-1-2) edge node [right] {$\cong$} (m-2-2)
(m-1-3) edge  (m-2-3)
(m-1-4) edge node [right] {$\cong$}  (m-2-4)
;
\path[->>]
(m-1-3.east|-m-1-4) edge  (m-1-4)
(m-2-3.east|-m-2-4) edge  (m-2-4)
;
\end{tikzpicture}
\end{equation}
Using the five lemma, we obtain that the morphism
  $H_2(M^\ell,\mathbb Z/n)\to H_2(\hat M,\mathbb Z/n)$ is an isomorphism.
\end{proof}

\vspace{.5cm}\noindent {\it Acknowledgements.} The authors thank
F. Petrov and K. Orr for discussions related to the subject of the
paper.

\end{document}